\documentclass[leqno,12pt]{amsart}

\newcommand{\barsymbolsetup}[1]{
	\newlength{\Xwidth}
	\settowidth{\Xwidth}{#1}
	\newlength{\Xheight}
	\settoheight{\Xheight}{#1}
}

\barsymbolsetup{$\mathcal{X}$}

\usepackage{enumerate}
\usepackage{amsmath}
\usepackage{mathrsfs}
\usepackage{color}
\usepackage{amssymb}

\usepackage{esint}

\usepackage[colorlinks=true,urlcolor=blue, citecolor=blue,linkcolor=blue,linktocpage,pdfpagelabels,bookmarksnumbered,bookmarksopen]{hyperref}
\usepackage{cancel} 

\numberwithin{equation}{section}

\newtheorem{theorem}{Theorem}[section]
\newtheorem{lemma}[theorem]{Lemma}
\newtheorem{corollary}[theorem]{Corollary}
\newtheorem{remark}[theorem]{Remark}
\newtheorem{definition}[theorem]{Definition}

\newtheorem{proposition}[theorem]{Proposition}

\newcommand{\R}{\mathbb{R}}

\begin{document}
	\title[Schatten--Lorentz properties of commutators with matrix weights]{Schatten--Lorentz properties of commutators with matrix weights}

\author[T. Hyt\"onen]{Tuomas Hyt\"onen}
\address{Aalto University, Department of Mathematics and Systems Analysis, P.O. Box 11100,
FI-00076 Aalto, Finland}
\email{tuomas.hytonen@aalto.fi}

\author[X. Quan]{Xingyan Quan}
\address{Aalto University, Department of Mathematics and Systems Analysis, P.O. Box 11100,
FI-00076 Aalto, Finland}
\email{quanxingyan@csu.edu.cn}

\thanks{T.H. was supported by the Research Council of Finland, grants 364208 and 371637. }



	\begin{abstract}
		We study the membership of commutators $[b,T]$ of Calder\'on--Zygmund
		operators in the Schatten--Lorentz classes in the matrix-weighted setting.
		Working on a space of homogeneous type $(X,\rho,\mu)$ with a matrix $A_2$
		weight $W$, we characterise, for $p\in(1,\infty)$ and $q\in[1,\infty)$, the
		Schatten--Lorentz norm $\|[b,T]\|_{S^{p,q}(L^2(W))}$ in terms of an
		oscillation (Besov-type) norm $\|b\|_{\mathrm{Osc}^{p,q}}$ of the symbol: for all $p\in(1,\infty)$, finiteness of $\|b\|_{\mathrm{Osc}^{p,q}}$
		is necessary; and for $p>p(\eta)$, where $p(\eta)$ depends on the regularity of the kernel of $T$ and a dimensional parameter of the space, the two-sided estimate
		$\|[b,T]\|_{S^{p,q}(L^2(W))}\approx\|b\|_{\mathrm{Osc}^{p,q}}$ holds.
		Extending a recent unified framework for unweighted and scalar-weighted versions of such results by one of us to the matrix-weighted setting, our approach combines the dyadic representation of singular integral operators with the nearly weakly orthogonal (NWO) sequences of Rochberg and
		Semmes.
	\end{abstract}

\maketitle

	\section{Introduction}

 Let $T$ be a singular integral operator and let $b$ be a locally integrable function. The commutator of $b$ with $T$ is defined by
 $$[b, T]f:=b\,T f-T(b f).$$
 The systematic study of such commutators originates in the seminal work of Coifman, Rochberg and Weiss \cite{Coifman1}, who showed that for the Riesz transforms $R$ on $\mathbb{R}^n$ the commutator $[b,R]$ is bounded on $L^p$ if and only if $b\in\mathrm{BMO}$. In the two-weight setting, Bloom \cite{bloom} characterised the boundedness of $[b,H]$, where $H$ denotes the Hilbert transform on $\mathbb{R}$, in terms of a weighted $\mathrm{BMO}$ space. Since then the weighted theory of commutators has developed extensively.

 A recent and active line of research seeks the matrix-weighted counterparts of these scalar results. A central issue is the correct formulation of the matrix $A_2$ condition. Fix $n,d\in\mathbb{N}$. A matrix weight on $\mathbb{R}^n$ is a locally integrable function $W:\mathbb{R}^n\to\mathbb{M}_d(\mathbb{C})$ whose values are self-adjoint and positive definite for almost every $x\in\mathbb{R}^n$. Following \cite{Treil1997}, $W$ is said to satisfy the matrix $A_2$ condition if
	\begin{equation}\label{matrix-A_2-intro}
		[W]_{A_2}:=\sup_{Q}\left\|\Big(\frac1{|Q|}\int_Q W\,dx\Big)^{\frac12}\Big(\frac1{|Q|}\int_Q W^{-1}\,dx\Big)^{\frac12}\right\|_\infty^2<\infty,
	\end{equation}
	the supremum being taken over all cubes $Q\subset\mathbb{R}^n$.
	Commutators in the matrix-weighted setting have also attracted considerable attention. Isralowitz, Kwon and Pott \cite{Isralowitz} extended Bloom's theorem to matrix weights, and Isralowitz, Pott and Treil \cite{Isralowitz2022} subsequently gave conditions that almost characterise the two-matrix-weight boundedness of commutators with Calder\'on--Zygmund operators for completely arbitrary matrix weights---a result that is new even in the fully scalar one-weight case.

 In this article we investigate a finer property of the commutator than its boundedness or compactness, namely its membership in the Schatten--Lorentz classes $S^{p,q}$ (defined in Subsection~\ref{Schatten}). These classes measure the rate of decay of the singular values of a compact operator and refine the Schatten ideals $S^p=S^{p,p}$. They play a significant role in operator theory, noncommutative analysis, and mathematical quantum theory, where the behaviour of compact operators on infinite-dimensional spaces is of central importance. In harmonic analysis, commutators of singular integral and multiplication operators may be regarded as generalisations of Hankel-type operators, whose Schatten-class membership is classical \cite{peller}. It is therefore natural to study  Schatten--Lorentz membership of $[b,T]$. This was started by Janson, Peetre and Wolff \cite{janson,janson2} and Rochberg
and Semmes \cite{Rochberg} in the 1980s. More recently, their results have been
extended to both weighted Euclidean \cite{GLW,Hyto,lacey,lacey3} and other
(e.g., \cite{fan1,Fan2,HW,Hyto,LXY,wei}) settings. Paper \cite{Hyto} appears on purpose in both lists: while addressing more general spaces, some of its results are already original in the weighted Euclidean space.

	Motivated by this, the aim of the present article is to establish a Schatten--Lorentz characterisation for commutators of Calder\'on--Zygmund operators with matrix-valued weights. Following the recent framework of \cite{Hyto}, we work on a space of homogeneous type $(X,\rho,\mu)$, but our results are already new for general Calder\'on--Zygmund operators $T$ in the Euclidean setting $X=\mathbb{R}^n$. While finishing this paper, we learned about a recent work of Li, Liang, Shen, and Wick \cite{LLSW},
where they characterize the Schatten--Lorentz membership
$[B,R_j]\in S^{p,q}(L^2(U),L^2(V))$ with two matrix weights $U,V$, where $B$ is a matrix-valued function and
$R_j$ is one of the classical Riesz transforms on $\R^n$. Our results and theirs intersect
when $X=\mathbb{R}^n$, $T=R_j$, $B=b$ is scalar-valued, and $U=V=W$. Outside this
common case, \cite{LLSW} deals with more general (matrix-valued) multipliers $B$
and pairs of weights $U,V$, while we deal with a broader class of operators $T$
and underlying spaces $X$. Perhaps oversimplifying a little, one may argue that a key point of \cite{LLSW} is identifying the new kinds of characterizing conditions that arise in their setting, while our emphasis is in showing that the exact same conditions that characterize the unweighted  $S^{p,q}$ norms of commutators in various recent contributions (placed under a common umbrella by the unified framework of \cite{Hyto}) also characterize their matrix-weighted $S^{p,q}$ norms in our setting with a  single matrix-weight $W$ and a scalar multiplier $b$.

Our main result is the following.
	
		\begin{theorem}\label{main}
Let $(X,\rho,\mu)$ be a space of homogeneous type with geometric dimension
$\Delta$ (Definition~\ref{dimension}), let $\mathscr D$ be a dyadic system on
$(X,\rho,\mu)$ (Definition~\ref{dyadiccube}), and let $W$ be a matrix $A_2$
weight. Let $T\in\mathscr L(L^2(\mu))$ be an operator with a non-degenerate
Calder\'on--Zygmund kernel with parameter $\eta\in(0,1]$, and
let $b\in L^1_{\mathrm{loc}}(\mu)$.

For every $q\in[1,\infty]$ the following estimates hold. First, for every
$p\in(1,\infty)$,
\[
    \|b\|_{\mathrm{Osc}^{p,q}}
    \lesssim
    \|[b,T]\|_{S^{p,q}(L^2(W))},
\]
where the norm $\|b\|_{\mathrm{Osc}^{p,q}}$ is defined in Definition \ref{def:Osc} and \eqref{osc-r}.
Conversely, if
\[
    p>p(\eta):=\max\Big\{1,\big(\frac{\eta}{\Delta}+\frac12\big)^{-1}\Big\},
\]
then
\[
    \|[b,T]\|_{S^{p,q}(L^2(W))}
    \lesssim
    \|b\|_{\mathrm{Osc}^{p,q}}.
\]
In particular, for $p>p(\eta)$ one has the two-sided estimate
\[
    \|[b,T]\|_{S^{p,q}(L^2(W))}
    \approx
    \|b\|_{\mathrm{Osc}^{p,q}}.
\]
\end{theorem}

The scalar counterpart of our results is from \cite{Hyto}, where the Schatten class $S^p$ (and the weak Schatten class $S^{d,\infty}$) membership of $[b,T]$ was characterized on spaces of homogeneous type, including the scalar weighted spaces $L^2(w)$ with a Muckenhoupt weight $w\in A_2$. We follow the dyadic-representation approach of \cite{Hyto} throughout, making explicit use of several of its structural results: the dyadic representation theorem,  the oscillation estimate.

Together with Proposition~1.29 in \cite{Hyto}, we obtain the following corollary.

\begin{corollary}\label{cor:classical}
Let $(X,\rho,\mu)$ and $T\in\mathscr L(L^2(\mu))$ be as in Theorem~\ref{main}.
Assume that $(X,\rho,\mu)$ has geometric dimension $\Delta$, lower dimension
$d$, and upper dimension $D$ in the sense of Definition~\ref{dimension}. Let
$b\in L^1_{\mathrm{loc}}(\mu)$. In items \textup{(ii)} and \textup{(iii)},
where $d>1$, we assume in addition that $(X,\rho)$ is a complete metric space
supporting a $(1,d)$-Poincar\'e inequality: for some constants $c_P>0$ and
$\lambda\in[1,\infty)$,
\begin{equation}\label{eq:poincare}
  \big\langle\big|f-\langle f\rangle_{B(x,r)}\big|\big\rangle_{B(x,r)}
  \leq
  c_P\,r\,\big\langle(\operatorname{lip} f)^d\big\rangle_{B(x,\lambda r)}^{1/d}
\end{equation}
for all $x\in X$, all $r>0$, and all Lipschitz functions $f$, where
$$
  \operatorname{lip} f(x)
  :=\liminf_{r\to0}\sup_{\rho(x,y)\leq r}\frac{|f(x)-f(y)|}{r}.
$$
\begin{enumerate}[\upshape(i)]
\item \textup{(Besov case.)}
For every $p\in(1,\infty)$, the finiteness of the homogeneous Besov norm
\[
 \|b\|_{\dot B^p(\mu)}
 :=\left(
 \iint_{X\times X}\frac{|b(x)-b(y)|^p}{V(x,y)^2}\,d\mu(x)\,d\mu(y)
 \right)^{1/p}
\]
is necessary for $[b,T]\in S^p(L^2(W))$. It is sufficient whenever
\[
 p>p(\eta):=\max\Big\{1,\big(\frac{\eta}{\Delta}+\frac12\big)^{-1}\Big\}.
\]
More precisely,
\[
 \|b\|_{\dot B^p(\mu)}\lesssim \|[b,T]\|_{S^p(L^2(W))},
 \qquad p\in(1,\infty),
\]
and
\[
 \|[b,T]\|_{S^p(L^2(W))}\lesssim \|b\|_{\dot B^p(\mu)},
 \qquad p>p(\eta).
\]

\item \textup{(Haj\l{}asz--Sobolev case.)}
If $d=D>1$ and $\eta>(1-d/2)_+$, then $[b,T]$ belongs to the weak Schatten class $S^{d,\infty}(L^2(W))$ if and
only if $b$ belongs to the homogeneous Haj\l asz--Sobolev space
$\dot M^{1,d}(\mu)$ defined by the norm
\begin{equation*}
\begin{split}
 \|b\|_{\dot M^{1,d}(\mu)}
 :=\inf\bigl\{\|h\|_{L^d(\mu)}:
 |b(x)-b(y)| \leq (h(x) &+h(y))\rho(x,y)\\ & \text{for a.e. }x,y\in X\bigr\}.
\end{split}
\end{equation*}
In this case,
\[
 \|[b,T]\|_{S^{d,\infty}(L^2(W))}
 \approx
 \|b\|_{\dot M^{1,d}(\mu)}.
\]

\item If $d>1$ and $p\in(0,d]$, then $[b,T]\in S^p(L^2(W))$ if and only if
$b$ is constant.
\end{enumerate}
\end{corollary}

The Haj\l asz--Sobolev space $\dot M^{1,d}(\mu)$ was introduced in \cite{Hajlasz}. In order to connect Corollary \ref{cor:classical} to know results in $\R^d$, we note in particular that $\dot M^{1,d}(\mu)=\dot W^{1,d}(\R^d)$ if $\mu$ is the Lebesgue measure on $\R^d$; this is \cite[Theorem 1]{Hajlasz}. More information about $\dot M^{1,d}(\mu)$ relevant to the present context can be found in \cite{Hyto}. 

 \begin{proof}[Proof of Corollary \ref{cor:classical}]
The equivalence between the dyadic oscillation scale and the homogeneous Besov scale is \cite[Proposition~1.29]{Hyto}. This gives item \textup{(i)} from Theorem~\ref{main} by taking $q=p$. Under the additional hypotheses $d=D>1$ and $\eta>(1-d/2)_+$, and using the completeness of $(X,\rho)$ together with the
$(1,d)$-Poincar\'e inequality \eqref{eq:poincare}, the same proposition with  \cite[Theorem 1.1]{Hyto1} identifies the critical weak space $\mathrm{Osc}^{d,\infty}$ with the homogeneous Haj\l{}asz--Sobolev space $\dot M^{1,d}(\mu)$, giving item \textup{(ii)}. Finally, item \textup{(iii)} is obtained as follows: For $p\in(0,d]$, we $[b,T]\in S^p(L^2(W))\subset S^d(L^2(W))$, after which we obtain from case (i) that $b\in\dot B^d(\mu)=\{\textrm{constants}\}$, where the last identity is part of \cite[Theorem 1.5]{Hyto1}.
\end{proof}

The paper is organised as follows. In Section~\ref{prelim} we collect the
preliminaries on spaces of homogeneous type, dyadic cubes, the oscillation spaces
$\mathrm{Osc}^{p,q}$, the Schatten--Lorentz classes, and matrix $A_2$ weights.
Section~\ref{NWO} develops the matrix-weighted NWO (nearly weakly orthogonal sequences) techniques---the bilinear-sublinear
maximal operator and the Carleson-type bounds---which form the technical core of the
paper. In Section~\ref{lower} we prove the lower bound (necessity) in Theorem~\ref{main},
which relies on Section~\ref{NWO}. The upper bound is then assembled
through the dyadic representation theorem: Section~\ref{para} treats commutators of
dyadic paraproducts, Section~\ref{shift} treats commutators of dyadic shifts, and
Section~\ref{upper} combines these via the representation theorem to complete the
proof of Theorem~\ref{main}.

	\section{Preliminaries}\label{prelim}
	
	\subsection{Spaces of homogeneous type}\label{homospace}
	Our results are set on spaces of homogeneous type, a very general framework. We collect the necessary notation and definitions here.
	Throughout the paper, $C$ denotes a positive constant, independent of the main parameters, which may change from line to line. We write $A \lesssim B$ to mean $A \leq C B$ for some $C>0$, and $A \sim B$ to mean $A \lesssim B$ and $B \lesssim A$. For $1 \leq p \leq \infty$, we denote by $p^{\prime}$ the H\"older conjugate of $p$, satisfying $\frac{1}{p}+\frac{1}{p^{\prime}}=1$.
	We denote the average of a function $f$ over a ball $B$ by
	$$
	\langle f\rangle_{B}:=\fint_{B} f(x)\, d\mu:=\frac{1}{\mu(B)} \int_{B} f(x)\, d\mu.
	$$
 We recall the three dimension concepts that appear in Theorem~\ref{main} and Corollary~\ref{cor:classical}.

 \begin{definition}\label{dimension}
 We say that a quasi-metric measure space $(X,\rho,\mu)$ has
 \begin{enumerate}
 \item geometric dimension $\Delta$ if every $B(x,R)$ can contain at most $C(R/r)^{\Delta}$ mutually $r$-separated points;
 \item lower dimension $d$ if $$\Big(\frac R r\Big)^d\lesssim\frac{V(x,R)}{V(x,r)};$$
 \item upper dimension $D$ if $$\frac{V(x,R)}{V(x,r)}\lesssim\Big(\frac R r\Big)^D$$
 \end{enumerate}
 in each case uniformly in $x\in X$ and $0<r\leq R<\infty$.
 \end{definition}

\begin{remark}
The geometric dimension of Definition \ref{dimension} was called {\em separation dimension} in \cite[Definition 3.1]{Hyto} as a distinction from a seemingly different {\em covering dimension}, also in \cite[Definition 3.1]{Hyto}. However, \cite[Proposition 3.6]{Hyto} shows that these two dimension concepts are equal. The word ``geometric'' emphasises the fact that this dimension concept only depends on the metric space $(X,\rho)$, while lower and upper dimension also refer to the measure $\mu$.
\end{remark}

	The considerations in the rest of this paper will rely on the notion of dyadic cubes.

\begin{definition}\label{dyadiccube}
	A system of dyadic cubes $\mathscr{D}$, or just a dyadic system, on the space of homogeneous type $(X, \rho, \mu)$ is a collection
	$$
	\mathscr{D}=\bigcup_{k \in \mathbb{Z}} \mathscr{D}_{k},
	$$
	where
	\begin{enumerate}
		\item each $\mathscr{D}_{k}$ is a measurable partition of $X$;
		\item each $\mathscr{D}_{k+1}$ refines the previous $\mathscr{D}_{k}$;
	\item for parameters $\delta \in(0,1)$ and $0<c_{0} \leq C_{0}<\infty$, each $Q \in \mathscr{D}_{k}$ is essentially a ball of size $\delta^{k}$, in the sense that, for some ``centre'' $z_{Q} \in X$,
	\begin{equation*}
		B\left(z_{Q}, c_{0} \delta^k\right) \subseteq Q \subseteq B(z_Q,C_0\delta^k). 
	\end{equation*}
	\end{enumerate}

	The dyadic system is said to be connected if, in addition,
	
	(4) for any two $P, Q \in \mathscr{D}$, there is some $R \in \mathscr{D}$ such that $P \cup Q \subseteq R$.

\end{definition}

We write $\ell(Q):=\delta^k$ for $Q\in\mathscr{D}_k$ and call it
the \emph{side length} of $Q$, and we set
\begin{equation}\label{eq:BQ}
   B_Q:=B(z_Q,C_0\ell(Q)),
\end{equation}
so that
$Q\subseteq B_Q$ and $\mu(B_Q)\lesssim\mu(Q)$, with constants depending only on
the space.

Recall that the standard dyadic system $\mathscr{D}=\left\{2^{-k}\left([0,1)^{d}+m\right): k \in \mathbb{Z}, m \in \mathbb{Z}^{d}\right\}$ on $\mathbb{R}^{d}$. The standard dyadic system  is not connected, but a simple modification of the construction yields a connected system.
In general, the same underlying set \(Q\) may belong to \(\mathscr{D}_k\) for several distinct values of \(k\). It is therefore
convenient to keep the level in the notation and write the pair $(Q,k)$ whenever
the level is not determined by the set $Q$. 

For $Q\in\mathscr{D}_k$, its \emph{children} are
\[
\operatorname{ch}(Q):=\{Q'\in\mathscr{D}_{k+1}:Q'\subseteq Q\},
\qquad
\operatorname{ch}^{(j)}(Q):=\operatorname{ch}\bigl(\operatorname{ch}^{(j-1)}(Q)\bigr),
\]
and, when the level is left unspecified, its strict children are
\[
[\operatorname{ch}](Q):=\{Q'\in\mathscr{D}:Q'\subsetneq Q
\ \text{and}\ Q'\subsetneq R\subsetneq Q\ \text{for no } R\in\mathscr{D}\},
\]
the maximal members of $\mathscr{D}$ properly contained in $Q$.
	
We recall the $\operatorname{Osc}^{p,q}$ space first introduced in \cite{Rochberg} on $X=\R^d$ and in \cite{Hyto} on general $X$ as here.


 \begin{definition}\label{def:Osc}
 Let $p,r\in(0,\infty)$ and $q\in(0,\infty]$.
A measurable function $b\in L^1_{\mathrm{loc}}(\mu)$ belongs to $\operatorname{Osc}_r^{p, q}$, if
 $$
		\|b\|_{\operatorname{Osc}_r^{p, q}}:=\left\|\left\{\inf_c\Big(\fint_{B_Q}|b-c|^r\,d\mu\Big)^{\frac1r}\right\}_{Q \in \mathscr{D}}\right\|_{\ell^{p, q}}<\infty,
		$$
 where $B_Q$ is the ball associated with $Q$ as in \eqref{eq:BQ}
\end{definition}

By \cite[Proposition 11.3 and Lemma 11.4]{Hyto}, for all $p\in(0,\infty)$ and $q\in(0,\infty]$, changing the third index $r\in(0,\infty)$ yields an equivalent quantity; accordingly, we define
\begin{equation}\label{osc-r}
  \|b\|_{\operatorname{Osc}^{p, q}}
  :=
 \|b\|_{\operatorname{Osc}_1^{p, q}}\approx
 \|b\|_{\operatorname{Osc}_r^{p, q}}.
\end{equation}

\subsection{Singular integral operators}\label{Singularintegral}
	We list the notation and definitions that will be used throughout. For a bounded linear operator $T \in \mathscr{L}\left(L^{2}(\mu)\right)$ with off-support integral representation
$$
T f(x)=\int_{X} K(x, y) f(y)\, d\mu(y), \quad x \in X \backslash \operatorname{supp} f,
$$
given by a kernel $K \in C(X \times X \setminus\{(x, x): x \in X\})$, we impose the following conditions.

\begin{definition}[Calder\'on--Zygmund kernel]\label{def:CZK}\mbox{}
\begin{enumerate}
\item The kernel $K$ is called an \emph{$\omega$-Calder\'on--Zygmund kernel} if it satisfies the Calder\'on--Zygmund size estimate
\begin{equation}\label{eq:CZ}
  	|K(x,y)|\lesssim \frac{1}{V(x,y)},
\end{equation}
	and the modulus-of-continuity bound
\begin{equation}\label{eq:CZomega}
	|K(x,y)-K(x',y)|+|K(y,x)-K(y,x')|\leq \omega\!\left(\frac{\rho(x,x')}{\rho(x,y)}\right)\frac{1}{V(x,y)}
\end{equation}
	for all $x,x',y\in X$ with $\rho(x,x')\ll\rho(x,y)$, where
	$${V(x, y)}:={\mu(B(x, \rho(x, y)))}.$$

In particular, given $\eta\in (0,1]$, we say that the kernel $K$ satisfies the Calder\'on--Zygmund estimates with parameter $\eta$ if it is an $\omega$-Calder\'on--Zygmund kernel with $\omega(t)=t^\eta$.
	\item The kernel $K$ is said to be \emph{non-degenerate} if for every $x \in X$ and $r \in(0, \infty)$ there exists a point $y \in X$ with $\rho(x, y) \sim r$ and
	\begin{equation*}
		|K(x, y)|+|K(y, x)| \gtrsim \frac{1}{V(x, y)},
	\end{equation*}		
	uniformly in $x$ and $r$.
\end{enumerate}
\end{definition}
	
\subsection{Schatten--Lorentz classes}\label{Schatten}
Let $H$ be a separable complex Hilbert space.
Let $x$ be a compact operator on $H$ with adjoint $x^*$. Then $|x|=\sqrt{x^*x}$ is a compact, self-adjoint, non-negative operator on $H$. Denote by $\left\{\mu_k(x)\right\}_{k\geq 1}$ the singular values of $x$, arranged in non-increasing order
 $$\mu_1(x)\geq \mu_2(x)\geq \cdots\geq \mu_k(x)\geq \cdots\geq 0,$$
 and each value is repeated according to its multiplicity.
 For $p\in(0,\infty)$ and $q\in(0,\infty]$, $x$ belongs to the \emph{Schatten--Lorentz class} $S^{p,q}(H)$ if the sequence $\left\{\mu_k(x)\right\}$ is $\ell^{p,q}$-summable; the class is equipped with the quasi-norm
 $$\|x\|_{S^{p,q}(H)}:=\Big(\sum_{k\geq 1} \mu_k(x)^{q}\, k^{q/p-1}\Big)^{1/q}\quad(q<\infty),\qquad \|x\|_{S^{p,\infty}(H)}:=\sup_{k\geq 1} k^{1/p}\mu_k(x).$$
 We write $S^{p,p}(H)=S^p(H)$ for the Schatten $p$-class. The trace of a trace-class operator $x$ is denoted by $\tau(x)$; for $x\geq 0$ one has $\tau(x)=\|x\|_{S^1(H)}=\sum_{k\geq 1}\mu_k(x)$. It is well known that
 \begin{equation}\label{s2-s1}
 \|x\|_{S^2(H)}^2=\||x|^2\|_{S^1(H)}=\tau(|x|^2)
 \end{equation}
 for every $x\in S^2(H)$.

 For vectors $e,f\in H$, we identify the elementary tensor $e\otimes f$ with the rank-one operator $h\mapsto \langle h,f\rangle e$.
 By the Schmidt decomposition, for every compact operator $x$ on $H$ there exist orthonormal systems $(e_n), (f_n)\subset H$ such that
 $$x=\sum_{n\geq 1}\mu_{n}(x)\, e_n\otimes f_n.$$

	\subsection{Matrix weights}
	 Given $1\leq N<\infty$ and $1\leq p<\infty$, the space $L^p\left(X; \mathbb{C}^N\right)$ consists of all measurable functions $f: X \to \mathbb{C}^N$ such that
	 	$$\|f\|_{L^p\left(X;\mathbb{C}^N\right)}:=\left(\int_{X}\|f(t)\|_{\mathbb{C}^N}^p \,d\mu \right)^{1/p}<\infty.$$
	 A \emph{matrix weight} $W$ on $X$ is a locally integrable function $W:X\to \mathbb{M}_N(\mathbb{C})$ whose values are almost everywhere positive definite.

The weighted Lebesgue space $L^p(W)$ consists of all measurable functions $f:X \to \mathbb{C}^N$ for which
 	$$\|f\|_{L^p(W)}:=\Big(\int_X \|W^{\frac{1}{p}}(t)f(t)\|_{\mathbb{C}^N}^p \,d\mu \Big)^{1/p}<\infty.$$	
We say that $W\in A_2$ if
	\begin{equation}\label{matrix-A_2}
		[W]_{A_2}:=\sup_{Q}\left\|\Big(\frac{1}{\mu(Q)}\int_Q W\,dx\Big)^{\frac{1}{2}}\Big(\frac{1}{\mu(Q)}\int_Q W^{-1}\,dx\Big)^{\frac{1}{2}}\right\|_\infty^2<\infty,
	\end{equation}
	where we write $\|\cdot\|_\infty$ for its operator norm, i.e., for $M\
	\in\mathbb M_N(\mathbb C)$, $\|M\|_\infty:=\sup_{\|v\|_{\mathbb{C}^N}=1}\|Mv\|_{\mathbb{C}^N}$.
Recall the reverse H\"older inequality (see, e.g., \cite[Proposition 2.1]{Christ2001}): if $W\in A_2$, then there exists $s>2$ such that for every cube $Q$,
	\begin{equation}\label{reverseholder}
	\int_Q\big\|\langle W^{-1}\rangle_Q^{\frac{1}{2}}W^{\frac{1}{2}}\big\|_\infty^s\,d\mu \lesssim_{[W]_{A_2}}\mu(Q).
	\end{equation}

	\section{Nearly weakly orthogonal sequences}\label{NWO}
	The nearly weakly orthogonal (NWO) sequences introduced by Rochberg and Semmes \cite{Rochberg} play a crucial role in establishing Schatten--Lorentz membership for compact operators on a Hilbert space: they provide a flexible substitute for the Schmidt decomposition that is well adapted to a dyadic structure. This section develops the matrix-weighted version of this notion, which is used in both the lower bound of Section~\ref{lower} and the upper bounds of Sections~\ref{para}--\ref{upper}.

	A parallel extension of \cite{Rochberg1} is also contained in \cite{LLSW}, but there are some differences in the details. Notably, we work over a general space of homogeneous type $X$, and we use the bisublinear formalism introduced in \cite{Hyto}, as in \eqref{M_e}.

	We begin by recalling the definition of an NWO sequence, following \cite[Definition 13.1]{Hyto}; the notion originates in \cite[p.241]{Rochberg1} and \cite[Definition~7.1]{Rochberg}.
	\begin{definition}Let $B_{Q}^{*}=B(x_Q,c\,\ell(Q))$ be a concentric enlargement of $B_Q$ by a factor $c$ independent of $Q$. Fix $r\in(2,\infty)$. A sequence $(e_Q)_{Q\in \mathscr{D}}$ is called an \emph{$\mathrm{NWO}(r)$ sequence} if, for every $Q\in \mathscr{D}$, $e_Q$ is supported on $B_Q^*$ and
	$$\|e_Q\|_r\lesssim \mu(Q)^{\frac{1}{r}-\frac{1}{2}}.$$
We call $(e_Q)_{Q\in\mathscr D}$ an \emph{NWO sequence} if it is an $\mathrm{NWO}(r)$ sequence for every $r\in(2,\infty)$; in other words,
\[
\mathrm{NWO}:=\bigcap_{r\in(2,\infty)}\mathrm{NWO}(r).
\]
The implied constant in the norm estimate may depend on $r$, but is independent of~$Q$.

\end{definition}
Let $e_Q,h_Q:X\to\mathbb C$ be scalar-valued NWO functions, $f,g\in L^2(\mu;\mathbb C^N)$ be vector-valued functions.
	For a pair sequence $\mathcal{E}=(e_Q,h_Q)_{Q\in \mathscr{D}}$ and a matrix-weight $W$, we consider the bilinear-sublinear maximal operator
	\begin{equation}\label{M_e}
	 M_{\mathcal{E},W}(f,g)(x):=\sup_{Q\in\mathscr{D}}\chi_{Q}(x)\frac{|\langle W^{-\frac{1}{2}}e_Q,f\rangle \cdot \langle W^{\frac{1}{2}}h_Q,g\rangle|}{\mu(Q)},
	\end{equation}
where the bracket
$
\langle e,f\rangle=\int_X\overline{e(y)}f(y)\,d\mu(y)$
 and 
  the  pairing $u\cdot v:\sum_{j=1}^N u_j\overline{v_j}$. Note that two brackets above are vectors-valued and pairing  is scalar-valued. Consequently, $M_{\mathcal E,W}(f,g)$ is a non-negative scalar-valued function. 
	We have the following result.
	
		\begin{lemma}\label{bimaximal}Let $(X,\rho,\mu)$ be a space of homogeneous type and $W$ a matrix $A_2$ weight. Let $\left\{e_{Q}\right\}_{Q \in \mathscr{D}}$ and $\left\{h_{Q}\right\}_{Q \in \mathscr{D}}$ be NWO sequences. Then $M_{\mathcal{E},W}$, defined in \eqref{M_e}, is bounded from $L^2(\mu;\mathbb C^N)\times L^2(\mu;\mathbb C^N)$ to $L^1(\mu)$.
	\end{lemma}
	
	\begin{proof} Choose $r>\frac{2s}{s-2}>2$, where $s$ is the exponent in \eqref{reverseholder} (depending only on $W$), so that $\|e_Q\|_r\lesssim \mu(Q)^{\frac{1}{r}-\frac{1}{2}}$ and $\|h_Q\|_r\lesssim \mu(Q)^{\frac{1}{r}-\frac{1}{2}}$.
 Inserting factors $\langle W^{-1}\rangle_Q^{\pm\frac{1}{2}}$ and $\langle W\rangle_Q^{\pm\frac{1}{2}}$ and applying the Cauchy--Schwarz inequality gives
			\begin{align*}
&{| \langle W^{-\frac12}e_Q,f\rangle\cdot \langle W^{\frac12}h_Q,g\rangle|}
			\\&={| \langle W\rangle_Q^{-\frac12}\langle W\rangle_Q^{\frac12}\langle W^{-\frac12}e_Q,f\rangle\cdot \langle W^{-1}\rangle_Q^{-\frac12}\langle W^{-1}\rangle_Q^{\frac12}\langle W^{\frac12}h_Q,g\rangle|}\\
			&\leq\big\|\langle W\rangle_Q^{-\frac12}\langle W^{-1}\rangle_Q^{-\frac12}\big\|\,{|\langle e_Q,\langle W\rangle_Q^{\frac12}W^{-\frac12}f\rangle| |\langle h_Q,\langle W^{-1}\rangle_Q^{\frac12}W^{\frac12}g\rangle|}. 
 \end{align*}
		Hence, by \eqref{matrix-A_2} and the Cauchy--Schwarz inequality,
		\begin{align*}
			&\|M_{\mathcal{E},W}(f,g)\|_{L^1(\mu)}
			\\&\lesssim \left\|\sup_{Q\in\mathscr{D}}\chi_{Q}(x)\frac{|\langle e_Q,\langle W\rangle_Q^{\frac{1}{2}}W^{-\frac{1}{2}}f\rangle| }{\mu(Q)^{\frac{1}{2}}}\right\|_{L^2(\mu)}\left\|\sup_{Q\in\mathscr{D}}\chi_{Q}(x)\frac{ |\langle h_Q,\langle W^{-1}\rangle_Q^{\frac{1}{2}}W^{\frac{1}{2}}g\rangle|}{\mu(Q)^{\frac{1}{2}}}\right\|_{L^2(\mu)}\\
 &=: \mathrm{I}\cdot\mathrm{II}.
		\end{align*}	
We estimate $\mathrm{II}$. By H\"older's inequality with $\frac{1}{\alpha}=1-\frac{1}{r}-\frac{1}{s}>1-\frac{s-2}{2s}-\frac{1}{s}=\frac{1}{2}$,
		\begin{align*}\mathrm{II}&\leq\left\|\sup_{Q\in\mathscr{D}}\chi_Q\frac{\big(\int_Q|\langle W^{-1}\rangle_Q^{\frac{1}{2}}W^{\frac{1}{2}}|^s\,d\mu\big)^{\frac{1}{s}}\|g\chi_{B_Q^*}\|_{\alpha}\|h_Q\|_r}{\mu(Q)^{\frac{1}{2}}}\right\|_{L^2(\mu)}\\
 &\lesssim_{[W]_{A_2}}\left\|\sup_{Q\in\mathscr{D}}\chi_Q\frac{\|g\chi_{B_Q^*}\|_{\alpha}}{\mu(Q)^{\frac{1}{\alpha}}}\right\|_{L^2(\mu)}\approx\|(M|g|^{\alpha})^{\frac{1}{\alpha}}\|_{L^2(\mu)}\lesssim \|g\|_{L^2(\mu)},
		\end{align*}
		where $M$ denotes the Hardy--Littlewood maximal operator; the final estimate uses the $L^2$-boundedness of $(M|g|^{\alpha})^{1/\alpha}$, which holds when $\alpha<2$.

		 The term $\mathrm{I}$ is handled identically: since $W^{-1}\in A_2$ if and only if $W\in A_2$ (see \eqref{matrix-A_2}), the same argument applies with $W$ and $W^{-1}$ interchanged, giving $\mathrm{I}\lesssim\|f\|_{L^2(\mu)}$. The claimed bound follows.
	\end{proof}

For a pair sequence $\mathcal{E}=(e_Q,h_Q)_{Q\in \mathscr{D}}$ of NWO sequences and a scalar sequence $(\lambda_Q)_{Q\in\mathscr{D}}$, define the operator
	$$T_{\mathcal{E}}:=\sum_{Q\in \mathscr{D}}\lambda_Q\, e_Q \otimes h_Q.$$
	Following \cite{Rochberg1,Rochberg}, we define the \emph{Carleson averaging operator} by
	$$\operatorname{Car}\lambda(P):=\frac{1}{\mu(P)}\sum_{Q\subset P}|\lambda_Q|\,\mu(Q),$$
 and say that $\lambda$ is a \emph{Carleson sequence} if $\operatorname{Car}\lambda\in \ell^\infty$.

	\begin{proposition}\label{T}Let $W$ be a matrix $A_2$ weight, let $\left\{e_{Q}\right\}_{Q \in \mathscr{D}}$ and $\left\{h_{Q}\right\}_{Q \in \mathscr{D}}$ be NWO sequences, and let $\lambda=(\lambda_Q)_{Q\in \mathscr{D}}$ be a Carleson sequence. Then
		\begin{equation}
			\|T_{\mathcal{E}}f\|_{L^2(W)}\lesssim \|\operatorname{Car} \lambda\|_\infty\, \|f\|_{L^2(W)}.
		\end{equation} 
	\end{proposition}
\begin{proof}
	Setting $g=W^{\frac{1}{2}}f$, the asserted estimate is equivalent to
\begin{equation}
		\|W^{\frac{1}{2}}T_{\mathcal{E}}W^{-\frac{1}{2}}g\|_{L^2(\mu,\mathbb{C}^N)}\lesssim \|\operatorname{Car}\lambda\|_\infty\,\|g\|_{L^2(\mu,\mathbb{C}^N)}.
\end{equation}
By duality it suffices to estimate, for $f,g\in L^2$,
\begin{align*}
	\sum_{Q\in \mathscr{D}}\lambda_Q |\langle e_Q,W^{-\frac{1}{2}}f\rangle \langle h_Q,W^{\frac{1}{2}}g\rangle|&=	\sum_{Q\in \mathscr{D}}\lambda_Q |\langle W^{-\frac{1}{2}}e_Q,f\rangle \langle W^{\frac{1}{2}}h_Q,g\rangle|
 \\	&\leq \|\operatorname{Car}\lambda\|_\infty\|M_{\mathcal{E},W}(f,g)\|_{L^1(\mu)},
\end{align*}
where the inequality follows from the Carleson embedding bound of \cite[Proposition~13.2]{Hyto}. The proof is completed by Lemma~\ref{bimaximal}.
\end{proof}
Adapting the arguments of \cite[Propositions 13.2 and 13.4, Corollary 13.6]{Hyto} to the matrix-weighted setting, we obtain the following.
\begin{corollary}\label{cor1} Let $W$ be a matrix $A_2$ weight, and let $\left\{e_{Q}\right\}_{Q \in \mathscr{D}}$ and $\left\{h_{Q}\right\}_{Q \in \mathscr{D}}$ be NWO sequences. Then for $p\in(0,\infty)$ and $q\in(0,\infty]$,
	\begin{equation}
		\|\sum_{Q\in \mathscr D}\lambda_Qe_Q \otimes h_Q\|_{S^{p,q}(L^2(W))}\lesssim \|M_{\mathcal E,W}\|_{L^2(\mu)\times L^2(\mu)\to L^1(\mu)}\|\lambda\|_{\ell^{p,q}}.
	\end{equation}
\end{corollary}
By duality between the $S^{p,q}$ and $\ell^{p,q}$ spaces, Corollary~\ref{cor1} yields the following dual estimate.
\begin{lemma}\label{cor14.3}
	Let $p\in(1,\infty)$, $q\in[1,\infty]$, and let $W$ be a matrix $A_2$ weight. Let $\left\{e_{Q}\right\}_{Q \in \mathscr{D}}$ and $\left\{h_{Q}\right\}_{Q \in \mathscr{D}}$ be NWO sequences. Then
	$$\|\{\langle Ae_Q,h_Q\rangle\}_{Q\in\mathscr{D}}\|_{\ell^{p,q}(\mathscr{D})}\lesssim\|M_{\mathcal{E},W}\|_{L^2(\mu)\times L^2(\mu)\to L^1(\mu)} \|A\|_{S^{p,q}(L^2(W))}.$$
\end{lemma}
\begin{proof}Let $\lambda=(\lambda_Q)_{Q\in \mathscr{D}}$ be a finitely supported sequence. By trace duality between $S^{p,q}$ and $S^{p',q'}$ followed by Corollary~\ref{cor1},
\begin{align*}
		\Big|\sum_{Q\in \mathscr{D}}\lambda_Q\langle Ae_Q,h_Q\rangle\Big|&=\Big|\tau\Big(W^{\frac{1}{2}}AW^{-\frac{1}{2}}\,W^{\frac{1}{2}}\sum_{Q\in \mathscr{D}}\lambda_Qe_Q\otimes h_Q\,W^{-\frac{1}{2}}\Big)\Big|\\
		&\leq\|W^{\frac{1}{2}}AW^{-\frac{1}{2}}\|_{S^{p,q}}\,\Big\|W^{\frac{1}{2}}\sum_{Q\in \mathscr{D}}\lambda_Qe_Q\otimes h_Q\,W^{-\frac{1}{2}}\Big\|_{S^{p',q'}}\\
		&\lesssim \|A\|_{S^{p,q}(L^2(W))}\,\|M_{\mathcal{E},W}\|_{L^2(\mu)\times L^2(\mu)\to L^1(\mu)}\,\|\lambda\|_{\ell^{p',q'}}.
	\end{align*}
	Taking the supremum over all finitely supported $\lambda$ with $\|\lambda\|_{\ell^{p',q'}}\leq 1$ and using  the duality of $\ell^{p,q}$ spaces, the claim follows.
\end{proof}

\section{The lower bound for Calder\'on--Zygmund commutators}\label{lower}
In this section we establish the lower bound in Theorem~\ref{main} using the NWO estimates from Section~\ref{NWO}. For each dyadic cube $Q$, we bound the mean oscillation $m_b(B_Q)$ by testing the commutator against suitable localised functions. The maximal-operator bound in Lemma~\ref{bimaximal} and the duality estimate in Lemma~\ref{cor14.3} then control the $\ell^{p,q}$ norm of these oscillations by the Schatten--Lorentz norm of the commutator.

\begin{proposition}\label{lowerbound}
Let $b \in L_{\mathrm{loc}}^{1}(\mu)$, and let $K$ be an $\omega$-Calder\'on--Zygmund kernel satisfying
\begin{equation*}
	\lim_{t \to 0} \omega(t)=0.
\end{equation*}
Let $p \in(1, \infty)$, $q \in[1, \infty]$, and $W \in A_{2}$. If $[b, T]$ is an operator with kernel $(b(x)-b(y)) K(x, y)$ and $[b, T] \in S^{p, q}\left(L^{2}(W)\right)$, then
$$
\|b\|_{\mathrm{Osc}^{p, q}} \lesssim\|[b, T]\|_{S^{p, q}\left(L^{2}(W)\right)}.
$$
\end{proposition}

\begin{proof} For each $Q \in \mathscr{D}$, we apply \cite[Proposition 15.5]{Hyto} with $E_{1}=B_{1}=B_{Q}$ to produce functions $f_{Q}^{j}$ and $g_{Q}^{j}$ such that
\begin{equation*}
	m_{b}\left(B_{Q}\right)=\frac{1}{\mu\left(B_{Q}\right)} \int_{B_{Q}}\left|b-\langle b\rangle_{Q}\right|\, d\mu \lesssim \sum_{j=1}^{2}\left|\left\langle[b, T] f_{Q}^{j}, g_{Q}^{j}\right\rangle\right|,
\end{equation*}
and
\begin{equation*}
	\left|f_{Q}^{j}\right|+\left|g_{Q}^{j}\right| \lesssim \frac{\mathbf{1}_{B_{Q}^{*}}}{\mu(Q)^{\frac{1}{2}}},
\end{equation*}
where $B_{Q}^{*}=c\cdot B_{Q}$ is a concentric dilation by a factor $c$ depending only on the implied constants in \cite[Proposition 15.5]{Hyto}, so that $B_{Q}^{*}$ contains both $E_{1}=B_{1}=B_{Q}$ and $E_{2}\subseteq B_{2}$. Summing over $Q\in\mathscr{D}$ and applying Lemma~\ref{bimaximal} and Lemma~\ref{cor14.3}, we obtain
\begin{align*}
\left\|\left\{m_{b}\left(B_{Q}\right)\right\}_{Q \in \mathscr{D}}\right\|_{\ell^{p, q}} &\lesssim \sum_{j=1}^{2}\|M_{\mathcal{E},W}\|_{L^2(\mu)\times L^2(\mu)\to L^1(\mu)} \|[b, T]\|_{S^{p, q}\left(L^{2}(W)\right)}
	\\
	&\lesssim\|[b, T]\|_{S^{p, q}\left(L^{2}(W)\right)},
\end{align*}
which is the desired bound.
\end{proof}

\section{Commutators of dyadic paraproducts}\label{para}
This section establishes the upper bound for the Schatten--Lorentz norm of commutators of dyadic paraproducts, which will be needed in the analysis of general Calder\'on--Zygmund operators in Section~\ref{upper}.

	For each dyadic cube $Q \in \mathscr{D}$, we define the martingale difference operators
$$
\mathbb{D}_{Q}:=\sum_{Q^{\prime} \in \operatorname{ch}(Q)} \mathbb{E}_{Q^{\prime}}-\mathbb{E}_{Q}, \quad [\mathbb{D}]_{Q}:=\sum_{Q^{\prime} \in [\operatorname{ch}](Q)} \mathbb{E}_{Q^{\prime}}-\mathbb{E}_{Q},
$$
where the first form implicitly uses the level information $Q=(Q,k)$, while the second depends only on $Q$ as a set.
For $b \in \operatorname{BMO}(\mathscr{D})$, the associated dyadic paraproduct is
$$
\Pi_{b}:=\sum_{k \in \mathbb{Z}} \sum_{Q \in \mathscr{D}_{k}} \mathbb{D}_{Q} b \otimes \frac{\chi_Q}{\mu(Q)}=\sum_{Q \in \mathscr{D}}[\mathbb{D}]_{Q}b \otimes \frac{\chi_Q}{\mu(Q)}
$$
 and the operators $\Gamma_b^\alpha$ are intermediate
paraproduct-type operators defined as 
$$\Gamma_b^\alpha:=\sum_{P\in\mathscr D}(\mathbb D_Pb)h_P^\alpha\otimes h_P^\alpha.$$
For $r\in (1,\infty)$ and $f\in L^r(X,\rho,\mu)$, we set
$$\|f\|_{\text{\L}^{r}(Q)}:=\Big(\frac{1}{\mu(Q)}\int_Q|f|^r\,d\mu\Big)^{\frac{1}{r}}.$$

The summands of $\Pi_b$ are rank-one tensors built from NWO sequences, so the Carleson-type estimate of Corollary~\ref{cor1} applies directly; the oscillation norm enters through the pointwise bound $\|[\mathbb{D}]_Q b\|_\infty\lesssim\|b-\langle b\rangle_Q\|_{\text{\L}^{1}(Q)}$.

\begin{proposition}\label{4.1}
Let $W\in A_2$, let $b\in L^1_{\mathrm{loc}}(\mu)$, and let
$\beta\in\mathrm{BMO}(\mathscr D)$. For all $p\in(0,\infty)$ and
$q\in(0,\infty]$,
\[
\begin{aligned}
 \|\Pi_b\|_{S^{p,q}(L^2(W))}
 &\lesssim \|b\|_{\mathrm{Osc}^{p,q}(\mathscr D)}, \\
 \|[b,\Pi_\beta]\|_{S^{p,q}(L^2(W))}
 &\lesssim \|b\|_{\mathrm{Osc}^{p,q}(\mathscr D)}\,\|\beta\|_{\mathrm{BMO}(\mathscr D)}, \\
 \|\Gamma_b^\alpha\|_{S^{p,q}(L^2(W))}
 &\lesssim \|b\|_{\mathrm{Osc}^{p,q}(\mathscr D)}, \\
 \|[b,\Gamma_\beta^\alpha]\|_{S^{p,q}(L^2(W))}
 &\lesssim \|b\|_{\mathrm{Osc}^{p,q}(\mathscr D)}\,\|\beta\|_{\mathrm{BMO}(\mathscr D)} .
\end{aligned}
\]
\end{proposition}
\begin{proof}
By \cite[Lemma~14.12]{Hyto}, each $\Gamma_b^\alpha$ is a finite linear
combination of paraproduct-type operators of the same form as $\Pi_b$. It is
therefore enough to prove the estimates for $\Pi_b$ and $[b,\Pi_\beta]$; the
proof for the corresponding $\Gamma_b^\alpha$ is identical.

		By definition of the weighted Schatten--Lorentz norm, $$	\left\|\Pi_{b}\right\|_{S^{p, q}\left(L^{2}(W)\right)}=	\left\|W^{\frac12}\Pi_{b}W^{-\frac12}\right\|_{S^{p,q}\left(L^2(\mu,\mathbb C^N)\right)}.$$
	We write
	$$
		W^{\frac12}\Pi_{b}W^{-\frac12}=W^{\frac12}\sum_{Q \in \mathscr{D}}[\mathbb{D}]_{Q} b \otimes \frac{\chi_Q}{\mu(Q)}W^{-\frac12}=\sum_{Q \in \mathscr{D}}\left\|[\mathbb{D}]_{Q} b\right\|_{\infty} W^{\frac12}e_{Q} \otimes h_{Q}W^{-\frac12},
		$$
	where
	$
	e_{Q}:=\frac{[\mathbb{D}]_{Q} b}{\left\|[\mathbb{D}]_{Q} b\right\|_{\infty} \mu^{\frac12}(Q)}
	$ and $
	h_{Q}:=\frac{\chi_Q }{ \mu(Q)^{\frac12}}
	$.
	For every $r\in(2,\infty)$,
	$$\|e_Q\|_r=\left\|\frac{[\mathbb{D}]_{Q} b}{\left\|[\mathbb{D}]_{Q} b\right\|_{\infty} \mu^{\frac12}(Q)}\right\|_r=\frac{\left\|[\mathbb{D}]_{Q} b\right\|_r}{\left\|[\mathbb{D}]_{Q} b\right\|_{\infty} \mu^{\frac12}(Q)}\leq \mu(Q)^{\frac1r-\frac12} $$
	and 	$$\|h_Q\|_r=\left\|\frac{\chi_Q }{ \mu(Q)^{\frac12}}\right\|_r= \mu(Q)^{\frac1r-\frac12},$$
	so that $(e_Q)_{Q\in\mathscr D}$ and $(h_Q)_{Q\in\mathscr D}$ are NWO sequences. Hence, by Corollary \ref{cor1} and Lemma~\ref{bimaximal},
	$$	\left\|\Pi_{b}\right\|_{S^{p, q}\left(L^{2}(W)\right)} \lesssim\left\|\left\{\left\|[\mathbb{D}]_{Q} b\right\|_{\infty}\right\}_{Q \in \mathscr{D}}\right\|_{\ell^{p, q}}.$$
		Moreover, by \cite[Lemma 19.2]{Hyto},
				$$
		\left\|[\mathbb{D}]_{Q} b\right\|_{\infty} \lesssim\left\|b-\langle b\rangle_{Q}\right\|_{\text{\L}^{1}(Q)}:=\fint_{Q}\left|b-\langle b\rangle_{Q}\right| \mathrm{d} \mu,
		$$
	and therefore
	$$
	\begin{aligned}
	\left\|\Pi_{b}\right\|_{S^{p, q}\left(L^{2}(W)\right)} 
		& \lesssim\left\|\left\{\left\|b-\langle b\rangle_{Q}\right\|_{\text{\L}^{1}(Q)}\right\}_{Q \in \mathscr{D}}\right\|_{\ell^{p, q}}=\|b\|_{\mathrm{Osc}^{p, q}(\mathscr{D})} .
	\end{aligned}
	$$
	This proves the first estimate.

	We turn to the commutator. Fix $r$ sufficiently large, as in Corollary~\ref{cor1}, and set
		$$
	e_{Q}^1:=\frac{\left(b-\langle b\rangle_{Q}\right)[\mathbb{D}]_{Q} \beta}{\left\|b-\langle b\rangle_{Q}\right\|_{\text{\L}^r(Q)}\|[\mathbb{D}]_{Q}\beta\|_\infty \mu(Q)^{\frac12}}, \quad h_{Q}^1:=\frac{\chi_Q}{\mu(Q)^{\frac12}}, 
	$$
	$$
	e_{Q}^2:=\frac{[\mathbb{D}]_{Q} \beta}{\|[\mathbb{D}]_{Q}\beta\|_\infty
		\mu(Q)^{\frac12}}, \quad h_{Q}^2:=\frac{\left(b-\langle b\rangle_{Q}\right) \chi_Q}{\left\|b-\langle b\rangle_{Q}\right\|_{\text{\L}^r(Q)}\mu(Q)^{\frac12} }, 
	$$
	$$
	\lambda^1_Q=\lambda^2_Q= \left\|b-\langle b\rangle_{Q}\right\|_{\text{\L}^{r}(Q)}\|[\mathbb{D}]_{Q}\beta\|_\infty. 
	$$
	The $r$-norms of $e_Q^{j}$ and $h_Q^{j}$ ($j=1,2$) are bounded by $\mu(Q)^{\frac1r-\frac12}$; hence each of these is an $\mathrm{NWO}(r)$ sequence.
	With this notation,
	\begin{align*}
		W^{\frac12}{\left[b, \Pi_{\beta}\right] } 	W^{-\frac12}& =\sum_{Q \in \mathscr{D}}	W^{\frac12}\left(b[\mathbb{D}]_{Q} \beta \otimes \frac{\chi_Q}{\mu(Q)}-[\mathbb{D}]_{Q} \beta \otimes \frac{\chi_Q b}{\mu(Q)}\right)	W^{-\frac12} \\
		& =\sum_{Q \in \mathscr{D}}	W^{\frac12}\left((b-\langle b\rangle_Q)[\mathbb{D}]_{Q} \beta \otimes \frac{\chi_Q}{\mu(Q)}-[\mathbb{D}]_{Q} \beta \otimes \frac{\chi_Q(b-\langle b\rangle_Q)}{\mu(Q)}\right)	W^{-\frac12} \\
		& =\sum_{Q \in \mathscr{D}}	W^{\frac12}(\lambda^1_Qe^1_Q\otimes h^1_Q-\lambda^2_Qe^2_Q\otimes h^2_Q)	W^{-\frac12}.
	\end{align*}
	Consequently, by Corollary~\ref{cor1},
	\begin{align*}
		\left\|\left[b, \Pi_{\beta}\right]\right\|_{S^{p, q}\left(L^{2}(W)\right)}&=\left\|W^{\frac12}\left[b, \Pi_{\beta}\right]W^{-\frac12}\right\|_{S^{p, q}\left(L^2(\mu,\mathbb C^N)\right)}\\
			& \lesssim\left\|\left\{\left\|b-\langle b\rangle_{Q}\right\|_{\text{\L}^{r}(Q)}\|[\mathbb{D}]_{Q}\beta\|_\infty\right\}_{Q \in \mathscr{D}}\right\|_{\ell^{p, q}} \\
				& \lesssim\left\|\left\{\left\|b-\langle b\rangle_{Q}\right\|_{\text{\L}^{r}(Q)}\right\}_{Q \in \mathscr{D}}\right\|_{\ell^{p, q}}\sup_{Q \in \mathscr{D}}\|[\mathbb{D}]_{Q}\beta\|_\infty \\
		& \lesssim\left\|\left\{\left\|b-\langle b\rangle_{Q}\right\|_{\text{\L}^{r}(Q)}\right\}_{Q \in \mathscr{D}}\right\|_{\ell^{p, q}}\|\beta\|_{\mathrm{BMO}} \\
		&=\|b\|_{\mathrm{Osc}_{r}^{p, q}(\mathscr{D})}\|\beta\|_{\mathrm{BMO}} \approx\|b\|_{\mathrm{Osc}^{p, q}(\mathscr{D})}\|\beta\|_{\mathrm{BMO}} ,
	\end{align*}
where the final equivalence is \eqref{osc-r}.

\end{proof}
	
	\section{Commutators for dyadic shifts}\label{shift}
	In this section we establish the upper bound for the Schatten--Lorentz norm of commutators of dyadic shifts. As with the paraproducts, this estimate feeds into the analysis of general Calder\'on--Zygmund operators in Section~\ref{upper}; the new feature here is that the bound depends quantitatively on the complexity $(i,j)$ of the shift through the geometric dimension $\Delta$, which is what ultimately produces the threshold $p(\eta)$. We first recall the relevant definitions.
	\begin{definition}\label{lab-S}
			A dyadic shift of complexity $(i, j)$ is an operator of the form
		\begin{equation*}
			\mathbb{S}=\sum_{k \in \mathbb{Z}} \sum_{R \in \mathscr{D}_{k}} A_{(R, k)}, \quad A_{(R, k)}=\sum_{\substack{P \in \operatorname{ch}^{i}(R, k) \\ Q \in \operatorname{ch}^{j}(R, k)}} a_{P Q R} h_{P} \otimes h_{Q} 
		\end{equation*}
		where $h_{P}, h_{Q}$ are cancellative Haar functions on their respective cubes $P, Q$, and $a_{P Q R}$ are coefficients. 
		The shift is said to be normalised if
	$$
	\left|a_{P Q R}\right| \lesssim \frac{\sqrt{\mu(P) \mu(Q)}}{\mu(R)}.
	$$
	\end{definition}

The main result of this section is the following commutator bound, whose dependence on the complexity $(i,j)$ is the source of the threshold $p(\eta)$ in Theorem~\ref{main}. Its proof, given at the end of the section, reduces---via the commutator identity for shifts---to a paraproduct estimate together with the weighted Schatten bound for a single shift (Proposition~\ref{square}).
\begin{theorem}\label{5.2}Let $(X, \rho, \mu)$ be a space of homogeneous type of geometric dimension $\Delta$, and let $W \in A_{2}$. Let $\mathbb{S}$ be a normalised dyadic shift of complexity $(i, j)$ on $L^{2}(\mu)$, and let $b \in L_{\text {loc }}^{1}(\mu)$. Then for all $p \in(0, \infty)$ and $q \in(0, \infty]$, we have
	$$
	\|[b, \mathbb{S}]\|_{S^{p, q}\left(L^{2}(W)\right)} \lesssim \delta^{-(i \wedge j) \Delta\left(\frac{1}{p}-\frac{1}{2}\right)_{++}}(1+i \wedge j)^{\frac{1}{p}}\|b\|_{\mathrm{Osc}^{p, q}},
	$$
	where $\delta \in(0,1)$ is the parameter from the dyadic system $\mathscr{D}$ and $$(\frac1p-\frac12)_{++}:=\begin{cases}1/p-1/2, &\text{ if } p<2,\\
 0, &\text{ if } p>2,\\
 \text {any  }\varepsilon>0, &\text{  if } p=2.
 \end{cases}$$ The factor $(1+i \wedge j)^{\frac{1}{p}}$ can be omitted, if $(X, \rho, \mu)$ has positive lower dimension $d>0$.

\end{theorem}
The estimate for a single shift will be obtained by realizing its components as integral operators and passing to their kernels. The following identity, relating the Hilbert--Schmidt norm of an integral operator to the $L^2$ norm of its kernel, is classical; we present a proof for completeness.
\begin{lemma}\label{kernel}Let $T\in S^{2}(L^{2}(\mu,\mathbb C^N))$ be a Hilbert--Schmidt operator with $\mathbb M_N(\mathbb C)$ matrix kernel $k(\cdot,\cdot)$, that is,
	$$Tf=\int k(\cdot,y)f(y)\,d\mu(y).$$
 Then $$\|k(x,y)\|_{L^{2}\left( \mu\times \mu, S^2(\mathbb C^N)\right)}^2=\|T\|_{S^{2}\left(L^{2}(\mu,\mathbb C^N)\right)}^2.$$
	\end{lemma}
\begin{proof}
	Let
	$$T=\sum_{n}\sigma_n e_n\otimes f_n$$
	 be the singular value decomposition of $T$, where $e_n(x),f_n(y)\in \mathbb C^N$.
	 Then $$k(x,y)=\sum_{n}\sigma_ne_n(x){f_n^*(y)},$$
	 where $^*$ denotes conjugate transpose.
	 Consequently,
	 \begin{align*}
	 \|k(x,y)\|_{S^2(\mathbb C^N)}^2&=\sum_{i,j}|\sum_{n}\sigma_ne_{n,i}(x){f_{n,j}}^*(y)|^2\\
	 &=\sum_{i,j}\sum_{n,m}\sigma_n\sigma_me_{m,i}(x){f_{m,j}}^*(y)\cdot{f_{n,j}}(y){e_{n,i}}^* (x)
	 ,
	 \end{align*}
 and therefore, by the orthonormality of $(e_n)$ and $(f_n)$,
 \begin{align*}
 	\|k(x,y)\|_{L^{2}\left(\mu\times \mu, S^2(\mathbb C^N)\right)}^2&=\sum_{m,n}\sigma_n\sigma_m\int e_m(x)\cdot{e_n^*(x)}dx\int {f_m(y)}^*\cdot f_n(y)dy\\&=\sum_{n}\sigma^2_n=\|T\|_{S^{2}\left(L^{2}(\mu,\mathbb C^N)\right)}^2.
 \end{align*}
\end{proof}
By Lemma~\ref{kernel}, the norm $S^{2}(L^{2}(\mu,\mathbb C^N))$  of an integral operator equals the norm $L^{2}\left(\mu\times \mu, S^2(\mathbb C^N)\right)$ of its kernel. This yields the following estimate for the components of a dyadic shift.
\begin{lemma}	For a matrix weight $W$, a component $A_{(R, k)}$ of a dyadic shift satisfies
		\begin{align*}
			\left\|A_{(R, k)}\right\|_{S^{2}\left(L^{2}(\mu,\mathbb C^N)\right)} &\lesssim\Big(\sum_{\substack{P \in \operatorname{ch}^{i}(R, k) \\ Q \in \operatorname{ch}^{j}(R, k)}}|a_{P Q R}|^2\frac1{\mu(P)}\int_P\tau(W^{\frac12}(x) \langle W^{-1}\rangle_QW^{\frac12}(x))d\mu(x) \Big)^{\frac{1}{2}}\\
			&=:\left\|a_{(R, k)}\right\|_{\ell^{2}_W} .
	\end{align*}
	\end{lemma}

\begin{proof} By definition of the weighted Schatten norm,
$$
\|T\|_{S^{2}\left(L^2(W)\right)}=\left\|W^{\frac{1}{2}} T W^{-\frac{1}{2}}\right\|_{S^{2}\left(L^2(\mu,\mathbb C^N)\right)}.
$$
 Hence, by Lemma~\ref{kernel},
$$
\begin{aligned}
	\left\|A_{(R, k)}\right\|_{S^{2}\left(L^{2}(W)\right)}^{2} & =\left\|W^{\frac{1}{2}} A_{(R, k)} W^{-\frac{1}{2}}\right\|_{S^{2}\left(L^{2}(\mu,\mathbb C^N)\right)}^{2} \\
	& =\iint_{R \times R}\Big\|\sum_{\substack{P \in \operatorname{ch}^{i}(R, k) \\
			Q \in \operatorname{ch}^{j}(R, k)}} a_{P Q R} W^{\frac{1}{2}} h_{P}(x) W^{-\frac{1}{2}} h_{Q}(y)\Big\|_{S^2}^{2} \mathrm{~d} \mu(x) \mathrm{~d} \mu(y).
\end{aligned}
$$
Since the cubes $P \in \operatorname{ch}^{i}(R, k)$ are pairwise disjoint, and likewise the cubes $Q$, combining this with the identity \eqref{s2-s1}, we obtain
\begin{align*}
&\Big\|\sum_{\substack{P \in \operatorname{ch}^{i}(R, k) \\
			Q \in \operatorname{ch}^{j}(R, k)}} a_{P Q R} W^{\frac{1}{2}} h_{P}(x) W^{-\frac{1}{2}} h_{Q}(y)\Big\|_{S^2(\mathbb C^N)}^2 \\
			&=\sum_{\substack{P \in \operatorname{ch}^{i}(R, k) \\
			Q \in \operatorname{ch}^{j}(R, k)}}\Big\| a_{P Q R} W^{\frac{1}{2}} h_{P}(x) W^{-\frac{1}{2}} h_{Q}(y)\Big\|_{S^2(\mathbb C^N)}^2
	\\
	&=\sum_{\substack{P \in \operatorname{ch}^{i}(R, k) \\
			Q \in \operatorname{ch}^{j}(R, k)}} |a_{P Q R}|^2 |h_{P}(x)h_{Q}(y) |^2
		\|W^{\frac12}(x) W^{-\frac12}(y)	\|_{S^2(\mathbb C^N)}^2 .
	\end{align*}
Since $\left|h_{P}\right| \lesssim \mu(P)^{-\frac{1}{2}} 1_{P}$, and similarly for the cubes $Q$, it follows that
\begin{align*}&\left\|A_{(R, k)}\right\|_{S^{2}\left(L^{2}(W)\right)}^{2} 
\\&=\sum_{\substack{P \in \operatorname{ch}^{i}(R, k) \\			Q \in \operatorname{ch}^{j}(R, k)}} |a_{P Q R}|^2\iint_{R \times R} h_{P}^2(x)\tau\left(		W^{\frac12}(x) W^{-1}(y)	W^{\frac12}(x) \right) h_{Q}^2(y) \mathrm{~d} \mu \mathrm{~d} \mu	\\&\lesssim\sum_{\substack{P \in \operatorname{ch}^{i}(R, k) \\
			Q \in \operatorname{ch}^{j}(R, k)}} |a_{P Q R}|^2 \iint_{R \times R} \frac{1_{P}(x)}{\mu(P)} \tau\left(
		W^{\frac12}(x) W^{-1}(y)	W^{\frac12}(x) \right) \frac{1_{Q}(y)}{\mu(Q)} \mathrm{~d} \mu \mathrm{~d} \mu \\
		& =\sum_{\substack{P \in \operatorname{ch}^{i}(R, k) \\
			Q \in \operatorname{ch}^{j}(R, k)}} |a_{P Q R}|^2\frac1{\mu(P)}\int_P\tau(W^{\frac12}(x) \langle W^{-1}\rangle_QW^{\frac12}(x))d\mu(x) .
\end{align*}
\end{proof}
For fixed complexity parameters $(i,j)$, we view the shift $\mathbb S=\mathbb S_a$ as a linear function of the parameter vector
\[
 a=\{a_{PQR}: R\in\mathscr D,\ k\in\mathcal K_{i,j}(R),\
 P\in\operatorname{ch}^i(R,k),\ Q\in\operatorname{ch}^j(R,k)\},
\]
where $\mathcal K_{i,j}(R)\subset\mathbb Z$ is a set of at most $1+\min\{i,j\}$ generations $k$ at which the cube $R$ can contribute a non-zero coefficient block $a_{(R,k)}$ (see \cite[Lemma~14.6]{Hyto}), that is,
\[
 \mathcal K_{i,j}(R)
 :=[k_{\min}(R)\vee(k_{\max}(R)-i\wedge j),\ k_{\max}(R)]\cap\mathbb Z,
\]
where
\[
\begin{aligned}
 k_{\min}(R)&:=\inf\{k\in\mathbb Z:R\in\mathscr D_k\},\\
 k_{\max}(R)&:=\sup\{k\in\mathbb Z:R\in\mathscr D_k\}.
\end{aligned}
\]

We also set
\[
 \mathscr D^*_{i,j}:=\{(R,k): R\in\mathscr D,\ k\in\mathcal K_{i,j}(R)\}.
\]
Adjusting the argument in \cite[Proposition~15.4]{Hyto} to the matrix-weighted setting, we obtain the following bound for a single shift, which combines a square-function estimate with the rank count governed by the geometric dimension $\Delta$. We also take the opportunity to fix a detail in \cite[Proposition 14.15 and 15.4]{Hyto}, where the claims concerning $p=2$ may have been slightly overoptimistic. 
\begin{proposition}\label{square}
Let $(X,\rho,\mu)$ be a space of homogeneous type of geometric dimension
$\Delta$. Let $\mathbb S_a$ be a dyadic shift of complexity $(i,j)$ on
$L^2(\mu)$. For all $p\in(0,\infty)$, $q\in(0,\infty]$, and $W\in A_2$,
$$
 \|\mathbb S_a\|_{S^{p,q}(L^2(W))}
 \lesssim
 \left(C\delta^{-(i\wedge j)\Delta}\right)^{(\frac1p-\frac12)_{++}}
 \|a\|_{\ell^{p,q}(\ell^2_W)},
$$
where
$$
 \|a\|_{\ell^{p,q}(\ell^2_W)}
 :=
 \Big\|\{\|a_{(R,k)}\|_{\ell^2_W}\}_{(R,k)\in\mathscr D^*_{i,j}}\Big\|_{\ell^{p,q}}.
$$

\end{proposition}

\begin{proof}

	The following identity, taken from \cite{Hyto}, rests on a square-function estimate of \cite{square2} that also holds for matrix $A_2$ weights. 
By an argument analogous to \cite[Proposition 14.15]{Hyto},
$$
\left\|\mathbb{S}_{a}\right\|_{S^{p}(L^2(W))}^{p}\sim
\sum_{(R, k) \in \mathscr{D}_{i, j}^{*}}\left\|A_{(R, k)}\right\|_{S^{p}(L^2(W))}^{p} .
$$
When $p \geq 2$, we dominate each component by its Hilbert--Schmidt norm:
$$
\left\|A_{(R, k)}\right\|_{S^{p}(L^2(W))} \leq\left\|A_{(R, k)}\right\|_{S^{2}(L^2(W))}\lesssim\left\|a_{(R, k)}\right\|_{\ell^{2}_W}, \quad p \geq 2.
$$
When $p<2$, by \cite[Lemma 20.14]{Hyto} the operator $A_{(R, k)}$ has rank at most $C \delta^{-(i \wedge j) \Delta}$, whence
$$
\begin{aligned}
	\left\|A_{(R, k)}\right\|_{S^{p}(L^2(W))} 
	& \leq \left\|W^{\frac12}A_{(R, k)}W^{-\frac12}\right\|_{S^{2}}\left(\operatorname{rank}W^{\frac12}A_{(R, k)}W^{-\frac12}\right)^{\frac{1}{p}-\frac{1}{2}}\\
		& \lesssim_d \left\|W^{\frac12}A_{(R, k)}W^{-\frac12}\right\|_{S^{2}}\left(\operatorname{rank}A_{(R, k)}\right)^{\frac{1}{p}-\frac{1}{2}}\\
	& \lesssim_d\left\|a_{(R, k)}\right\|_{\ell^{2}_W}\left(C \delta^{-(i \wedge j) \Delta}\right)^{(\frac{1}{p}-\frac{1}{2})_+}, \quad p<2.
\end{aligned}
$$
Combining the two cases gives
$$
\left\|A_{(R, k)}\right\|_{S^{p}(L^2(W))} \lesssim_d\left\|a_{(R, k)}\right\|_{\ell^{2}_W}\left(C \delta^{-(i \wedge j) \Delta}\right)^{\left(\frac{1}{p}-\frac{1}{2}\right)_{+}},
$$
and raising to the $p$-th power and summing over $(R, k) \in \mathscr{D}_{i, j}^{*}$ yields
\begin{equation*}
	\left\|\mathbb{S}_{a}\right\|_{S^{p}(L^2(W))} \lesssim_d\|a\|_{\ell^{p}\left(\ell^{2}_W\right)}\left(C \delta^{-(i \wedge j) \Delta}\right)^{\left(\frac{1}{p}-\frac{1}{2}\right)_{+}}=: C_{p}\|a\|_{\ell^{p}\left(\ell^{2}_W\right)} .
\end{equation*}
The general case $q\in(0,\infty]$ now follows by interpolation.
Choose
\[
0<p_0<p<p_1<\infty,\qquad
\frac1p=\frac{1-\theta}{p_0}+\frac{\theta}{p_1},
\]
and let $(\cdot,\cdot)_{\theta,q}$ denote the real interpolation functor.
By
\[
S^{p,q}=(S^{p_0},S^{p_1})_{\theta,q}
\qquad\text{\cite[Th\'eor\`eme~1]{merucci}},
\]
and
\[
\ell^{p,q}(\ell_W^2)
=
(\ell^{p_0}(\ell_W^2),\ell^{p_1}(\ell_W^2))_{\theta,q}
\qquad\text{\cite[Proposition~9]{kree1967}},
\]
together with the exactness of real interpolation
\cite[Theorem~3.11.2]{BL1976},
we obtain
\[
\|\mathbb S_a\|_{S^{p,q}}
\le
C_{p_0}^{\,1-\theta}
C_{p_1}^{\,\theta}
\|a\|_{\ell^{p,q}(\ell_W^2)}.
\]

If $p>2$, choose $2<p_0<p<p_1$, so that
$C_{p_0}=C_{p_1}=C$, giving
\[
\|\mathbb S_a\|_{S^{p,q}}
\lesssim
\|a\|_{\ell^{p,q}(\ell_W^2)}.
\]

If $0<p<2$, choose $p_0,p_1<2$. Since
\[
\left(\frac1{p_0}-\frac12\right)(1-\theta)
+
\left(\frac1{p_1}-\frac12\right)\theta
=
\frac1p-\frac12,
\]
we conclude
\[
\|\mathbb S_a\|_{S^{p,q}}
\lesssim
\bigl(C\delta^{-(i\wedge j)\Delta}\bigr)^{\frac1p-\frac12}
\|a\|_{\ell^{p,q}(\ell_W^2)}.
\]

Finally, when $p=2$, interpolate between some
$p_0<2$ and $p_1>2$. Writing
$1/p_0=1/2+\varepsilon$,
the previous estimate yields
\begin{align}\label{eq:eps-power}
\|\mathbb S_a\|_{S^{2,q}}
\lesssim
\delta^{-(i\wedge j)\Delta\varepsilon}
\|a\|_{\ell^{2,q}(\ell_W^2)},
\end{align}where $\varepsilon$ can be taken as small as we like.
This completes the proof.

\end{proof}

\begin{remark}  At the end of the proof of \cite[Proposition 14.15]{Hyto}, the limit $\varepsilon\to 0$ was taken. This may not be allowed, since the implicit constant in \eqref{eq:eps-power} may depend on $\varepsilon$ and may blow up as $\varepsilon \to 0$.
\end{remark}
 
Let $\mathbb{S}_{\langle b\rangle}$ be the new shift associated with the commutator $[b, \mathbb{S}]$ of $\mathbb{S}$ with $b\in L_{\mathrm{loc}}^{1}(\mu)$ (see \cite[Lemma 14.9]{Hyto});
 its components take the form
\begin{equation}\label{apqr}
	a_{P Q R}^{\prime}:=a_{P Q R}\left(\langle b\rangle_{P}-\langle b\rangle_{Q}\right), \quad\left|a_{P Q R}\right| \lesssim \frac{\sqrt{\mu(P) \mu(Q)}}{\mu(R)} . 
\end{equation}
The following result is the matrix-weighted analogue of \cite[Corollary 14.21]{Hyto}.
\begin{corollary}\label{5.6}
	Let $(X, \rho, \mu)$ be a space of homogeneous type with geometric dimension $\Delta$, and let $W \in A_{2}$. Let $\mathbb{S}$ be a normalised dyadic shift on $L^{2}(\mu)$, and let $\mathbb{S}_{\langle b\rangle}$ be the shift associated with coefficients $a'_{PQR}$ as in \eqref{apqr}. Then for all $p \in(0, \infty)$ and $q \in(0, \infty]$,
	$$
	\left\|\mathbb{S}_{\langle b\rangle}\right\|_{S^{p, q}\left(L^{2}(W)\right)} \lesssim\left(C \delta^{-(i \wedge j) \Delta}\right)^{\left(\frac{1}{p}-\frac{1}{2}\right)_{++}}(1+i \wedge j)^{\frac{1}{p}}\|b\|_{\mathrm{Osc}^{p, q}(\mathscr{D})}.
	$$
		The factor $(1+i \wedge j)^{\frac{1}{p}}$ can be omitted if $(X, \rho, \mu)$ has positive lower dimension.
\end{corollary}

 \begin{proof}
As in Proposition~\ref{square}, we find
$$
\begin{aligned}
	\left\|a_{(R, k)}^{\prime}\right\|_{\ell_W^2 } & \lesssim \left(\sum_{\substack{P \in \operatorname{ch}^{i}(R, k) \\
			Q \in \operatorname{ch}^{j}(R, k)}} \tau(
			\langle W\rangle_P \langle W^{-1}\rangle_Q)\frac{\mu(P)\mu(Q)}{\mu(R)^2}\left|\langle b\rangle_{P}-\langle b\rangle_{R}\right|^{2}\right)^{\frac12} \\
	+&\left(\sum_{\substack{P \in \operatorname{ch}^{i}(R, k) \\
			Q \in \operatorname{ch}^{j}(R, k)}} \tau(
			\langle W\rangle_P \langle W^{-1}\rangle_Q)\frac{\mu(P)\mu(Q)}{\mu(R)^2}\left|\langle b\rangle_{Q}-\langle b\rangle_{R}\right|^{2}\right)^{\frac12}\\
		&=: \text{I}+\text{I}\text{I} .
\end{aligned}
$$
For term I, we have
\begin{align*}
	\text{I}^2&=\sum_{\substack{P \in \operatorname{ch}^{i}(R, k) \\
			Q \in \operatorname{ch}^{j}(R, k)}} \int_P\int_Q\tau(W^{\frac12}(x) W^{-1}(y)W^{\frac12}(x))d\mu(x)d\mu(y)\frac{1}{\mu(R)^2}\left|\langle b\rangle_{P}-\langle b\rangle_{R}\right|^{2}\\
				&=\sum_{\substack{P \in \operatorname{ch}^{i}(R, k) }} \int_P\tau( W^{\frac12}(x)\langle W^{-1}\rangle_R W^{\frac12}(x))d\mu(x)\frac{1}{\mu(R)}\left|\langle b\rangle_{P}-\langle b\rangle_{R}\right|^{2}.
		\end{align*}
	Since $\tau( W^{\frac12}(x)\langle W^{-1}\rangle_R W^{\frac12}(x))=\tau( |W^{\frac12}(x)\langle W^{-1}\rangle_R^{\frac12}|^2)\lesssim_d\|W^{\frac12}(x)\langle W^{-1}\rangle_R^{\frac12}\|^2_{\infty},$ it follows that 
	\begin{align*}\text{I}^2
		&\lesssim_d\sum_{\substack{P \in \operatorname{ch}^{i}(R, k) }} \int_P \|W^{\frac12}(x)\langle W^{-1}\rangle_R^{\frac12}\|^2_{\infty}d\mu(x)\frac{1}{\mu(R)}\left|\langle b\rangle_{P}-\langle b\rangle_{R}\right|^{2}\\
		&=\sum_{\substack{P \in \operatorname{ch}^{i}(R, k) }} \frac1{\mu(P)}\int_P \|W^{\frac12}(x)\langle W^{-1}\rangle_R^{\frac12}\|^2_{\infty}d\mu(x)\frac{\mu(P)}{\mu(R)}\left|\langle b\rangle_{P}-\langle b\rangle_{R}\right|^{2}.
\end{align*}
Note that \begin{align}\label{bP}\mu(P)\Big(\frac1{\mu(P)}\int_P \|W^{\frac12}(x)\langle W^{-1}\rangle_R^{\frac12}\|^{2}_{\infty}d\mu(x)\Big)^r&\leq\mu(P)\frac1{\mu(P)}\int_P \|W^{\frac12}(x)\langle W^{-1}\rangle_R^{\frac12}\|^{2r}_{\infty}d\mu(x)\nonumber\\
	&= \int_P \|W^{\frac12}(x)\langle W^{-1}\rangle_R^{\frac12}\|^{2r}_{\infty}d\mu(x)
	\end{align}
and similarly 
\begin{equation}\label{br}\mu(P)\left|\langle b\rangle_{P}-\langle b\rangle_{R}\right|^{2 r^{\prime}}\leq \int_P|b-\langle b\rangle_{R}|^{2 r^{\prime}}. 
	\end{equation}
For $r\in(1,\infty)$, H\"older's inequality gives
\[
\begin{aligned}
\mathrm I^2
&\lesssim_d
\left(\sum_{P\in\operatorname{ch}^i(R,k)}
 \frac{\mu(P)}{\mu(R)}
 |\langle b\rangle_P-\langle b\rangle_R|^{2r'}\right)^{1/r'}\\
&\qquad{}\times
\left(\sum_{P\in\operatorname{ch}^i(R,k)}
 \frac{\mu(P)}{\mu(R)}
 \left(\fint_P
 \|W^{1/2}(x)\langle W^{-1}\rangle_R^{1/2}\|_\infty^2
 \,d\mu(x)\right)^r\right)^{1/r}.
\end{aligned}
\]
Applying \eqref{br} to the first factor and \eqref{bP} to the second, and using that the cubes in $\operatorname{ch}^i(R,k)$ partition $R$ up to a null set, we obtain
\[
\begin{aligned}
\mathrm I^2
&\lesssim_d
 \left(\fint_R |b-\langle b\rangle_R|^{2r'}\,d\mu\right)^{1/r'}
 \left(\fint_R
 \|W^{1/2}(x)\langle W^{-1}\rangle_R^{1/2}\|_\infty^{2r}
 \,d\mu(x)\right)^{1/r}.
\end{aligned}
\]
Choose $2r=s>2$ so that the reverse H\"older inequality \eqref{reverseholder} applies to $W^{-1}$.  Hence
\[
\mathrm I^2\lesssim_{[W]_{A_2}}
 \left(\fint_R |b-\langle b\rangle_R|^{2r'}\,d\mu\right)^{1/r'}.
\]

Since the same argument applies to the term II with the roles of $W$ and $W^{-1}$ reversed, we obtain the claimed
$$
 \mathrm{II}^2 \lesssim_{[W]_{A_2}} \left(\frac1{\mu(R)}\int_R|b-\langle b\rangle_R|^{2r'}\,d\mu\right)^{\!1/r'}.
$$
Combining the bounds on $\mathrm{I}$ and $\mathrm{II}$, we arrive at
\[
 \|a'_{(R,k)}\|_{\ell^2_W}
 \leq \mathrm{I}+\mathrm{II}
 \lesssim_{[W]_{A_2}}
 \left(\fint_R |b-\langle b\rangle_R|^{2r'}\,d\mu\right)^{\!1/(2r')}
 =: c_R.
\]
Now we apply Proposition~\ref{square} to the shift $\mathbb{S}_{\langle b\rangle}=\mathbb{S}_{a'}$:
\[
 \|\mathbb{S}_{\langle b\rangle}\|_{S^{p,q}(L^2(W))}
 \leq \left(C\delta^{-(i\wedge j)\Delta}\right)^{(1/p-1/2)_{++}}
 \|\{c_R\}_{(R,k)\in\mathscr{D}^*_{i,j}}\|_{\ell^{p,q}}.
\]
Each value $c_R$ appears in $\|\cdot\|_{\ell^{p,q}}$ with multiplicity $\#\mathcal{K}_{i,j}(R)\leq \min(i,j)+1$ (cf.\ the definition of $\mathscr{D}^*_{i,j}$). By a standard repetition estimate for Lorentz norms \cite[Lemma F.3.4]{HvNVW2}, this multiplicity contributes a factor of $(\min(i,j)+1)^{1/p}$, so that
\[
 \|\{c_R\}_{(R,k)\in\mathscr{D}^*_{i,j}}\|_{\ell^{p,q}}
 \leq (\min(i,j)+1)^{1/p}
 \|\{c_R\}_{R\in\mathscr{D}}\|_{\ell^{p,q}}.
\]
Since $\|\{c_R\}_{R\in\mathscr{D}}\|_{\ell^{p,q}}=\|b\|_{\mathrm{Osc}^{p,q}_{2r'}(\mathscr{D})}\approx\|b\|_{\mathrm{Osc}^{p,q}(\mathscr{D})}$ by \eqref{osc-r}, we conclude that
\[
 \|\mathbb{S}_{\langle b\rangle}\|_{S^{p,q}(L^2(W))}
 \lesssim \left(C\delta^{-(i\wedge j)\Delta}\right)^{(1/p-1/2)_{++}}(1+i\wedge j)^{1/p}
 \|b\|_{\mathrm{Osc}^{p,q}(\mathscr{D})},
\]
as claimed. If $(X,\rho,\mu)$ has positive lower dimension $d>0$, then by Lemma~\ref{dyadiccube} each $Q\in\mathscr{D}$ appears in exactly one level $\mathscr{D}_k$, so the repetition count is $1$ and the factor $(1+i\wedge j)^{1/p}$ is absent.
\end{proof}

\begin{proof}[Proof of Theorem~\ref{5.2}]
By \cite[Lemma~14.9]{Hyto}, the commutator of a dyadic shift $\mathbb S$ with
the pointwise multiplier $b$ admits the decomposition
\[
 [b,\mathbb S]
 =
 [\Pi_b,\mathbb S]
 +
 \sum_{1\leq \alpha<M}[\Gamma_b^\alpha,\mathbb S]
 +
 \mathbb S_{\langle b\rangle},
\]
where $\mathbb S_{\langle b\rangle}$ is the shift with coefficients
\[
 a'_{PQR}=a_{PQR}(\langle b\rangle_P-\langle b\rangle_Q).
\]

We first estimate the commutators on the right-hand side. Since $S^{p,q}$ is a two-sided
operator ideal and $\mathbb S$ is bounded on $L^2(W)$ and  the number of indices $\alpha$ is finite and depends only on the structural
constants, Proposition~\ref{4.1}
gives
$$
 \|[\Pi_b,\mathbb S]\|_{S^{p,q}(L^2(W))}+\sum_{1\leq\alpha<M}
 \|[\Gamma_b^\alpha,\mathbb S]\|_{S^{p,q}(L^2(W))}
  \lesssim
 \|b\|_{\mathrm{Osc}^{p,q}(\mathscr D)}.
$$

Finally, Corollary~\ref{5.6} yields
\[
 \|\mathbb S_{\langle b\rangle}\|_{S^{p,q}(L^2(W))}
 \lesssim
 \delta^{-(i\wedge j)\Delta(\frac1p-\frac12)_{++}}
 (1+i\wedge j)^{1/p}
 \|b\|_{\mathrm{Osc}^{p,q}(\mathscr D)}.
\]
Combining the preceding estimates and absorbing the paraproduct-type terms into
the right-hand side gives the claimed estimate, since
\[
 \delta^{-(i\wedge j)\Delta(\frac1p-\frac12)_{++}}(1+i\wedge j)^{1/p}\geq 1.
\]
If $(X,\rho,\mu)$ has positive lower dimension, the factor
$(1+i\wedge j)^{1/p}$ is omitted exactly as in Corollary~\ref{5.6}.
\end{proof}

\section{The upper bound for Calder\'on--Zygmund commutators}\label{upper}
In this section we establish the upper bound for the Schatten--Lorentz norm of a general Calder\'on--Zygmund commutator. The idea is to decompose $T$ by means of the dyadic representation theorem and to apply the paraproduct and shift estimates obtained in Sections \ref{para} and \ref{shift}.
To begin with, we recall the dyadic representation theorem from \cite{HytonenA2} (see also \cite[\S16]{Hyto} for the formulation used here).
\begin{theorem}[Dyadic representation; \cite{HytonenA2,Hyto}]\label{representation}
 Let $(X, \rho, \mu)$ be a space of homogeneous type, and let $T \in \mathscr{L}\left(L^{2}(\mu)\right)$ have an $\omega$-Calder\'on--Zygmund kernel. Then $T$ has a representation, in the sense of bilinear forms on $L^{2}(\mu)$,
 \begin{align*}
		T=T_{1} & +\mathbb{E}\left(\Pi_{T 1}+\Pi_{T^{*} 1}^{*}\right) \\
		& +\mathbb{E} \sum_{m=0}^{\infty} \omega\left(\delta^{m}\right)\left\{\mathbb{S}_{m, m}+\sum_{i=0}^{m-1}\left(\mathbb{S}_{m, i}+\mathbb{S}_{i, m}\right)+\sum_{j=1}^{m}\left(\mathbb{S}_{0, j}+\mathbb{S}_{j, 0}\right)\right\},
 \end{align*}
 	where
 	\begin{enumerate}
 		\item $T_{1}$ is an operator of rank at most one given by	$$	T_{1}= \begin{cases}\mu(X)^{-2}\langle T1,1\rangle 1 \otimes 1, & \text { if } \mu(X)<\infty \\ 0, & \text { if } \mu(X)=\infty.\end{cases}	$$
 		\item Each $\mathbb{S}_{m, n}$ is a (sum of $O(1)$ many) normalised dyadic shift(s) of complexity $(m, n)$ in the sense of Definition~\ref{lab-S}.
 		\item Each $\Pi_{b}$ is a dyadic paraproduct with symbol $b\in\{T1,T^*1\}$, i.e.,	$$	\Pi_{b}=\sum_{Q \in \mathscr{D}} \mathbb{D}_{Q} b \otimes \frac{\chi_Q}{\mu(Q)}.	$$
 	\end{enumerate}
\end{theorem}

\begin{lemma}\label{T1lemma}
Let $W\in A_2$, $b\in L^1_{\mathrm{loc}}(\mu)$, $p\in(0,\infty)$, and
$q\in(0,\infty]$. Then
\[
\|[b,T_1]\|_{S^{p,q}(L^2(W))}
\lesssim
\|b\|_{\mathrm{Osc}^{p,q}}.
\]
\end{lemma}
\begin{proof}
If $\mu(X)=\infty$, then $T_1=0$. Assume therefore that $\mu(X)<\infty$. Then
\[
[b,T_1]=\frac{\langle T1,1\rangle}{\mu(X)^2}(b\otimes 1-1\otimes b)
=\frac{\langle T1,1\rangle}{\mu(X)^2}\bigl((b-\langle b\rangle_X)\otimes 1
-1\otimes (b-\langle b\rangle_X)\bigr).
\]
Since this operator has rank at most two, its $S^{p,q}$ quasi-norm is controlled,
up to a constant depending only on $p$ and $q$, by its Hilbert--Schmidt norm. We
estimate the first rank-one term; the second is identical with $W$ and $W^{-1}$
interchanged. By Lemma~\ref{kernel},
\begin{align*}
&\left\|W^{1/2}\frac{\langle T1,1\rangle}{\mu(X)^2}
(b-\langle b\rangle_X)\otimes 1\, W^{-1/2}\right\|_{S^2(L^2(\mu,\mathbb C^N))}^2
\\
&\qquad\lesssim
\tau\left(
\frac1{\mu(X)}\int_X |b-\langle b\rangle_X|^2 W\,d\mu\,
\frac1{\mu(X)}\int_X W^{-1}\,d\mu
\right),
\end{align*}
where we used
$|\langle T\mu(X)^{-1/2},\mu(X)^{-1/2}\rangle|\leq \|T\|_{\mathscr L(L^2(\mu))}$.
Choosing $r>1$ so that the reverse H\"older inequality applies and using
H\"older's inequality, the last trace is bounded by
\[
\left(\fint_X |b-\langle b\rangle_X|^{2r'}\,d\mu\right)^{1/r'}
\left(\fint_X \|W^{1/2}\langle W^{-1}\rangle_X^{1/2}\|_\infty^{2r}\,d\mu\right)^{1/r}
\lesssim
\left(\fint_X |b-\langle b\rangle_X|^{2r'}\,d\mu\right)^{1/r'}.
\]
When $\mu(X)<\infty$, the connected dyadic system contains $X$ as a top cube, and
therefore the final expression is controlled by $\|b\|_{\mathrm{Osc}^{p,q}}^2$,
using \eqref{osc-r}. Taking square roots and treating the second rank-one term in
the same way proves the lemma.
\end{proof}

\begin{proposition}\label{upperbound}
Let $(X,\rho,\mu)$, $W$, $T$, and $b$ be as in Theorem~\ref{main}. If
$q\in[1,\infty]$ and
\[
    p>p(\eta):=\max\Big\{1,\big(\tfrac{\eta}{\Delta}+\tfrac12\big)^{-1}\Big\},
\]
then
\[
\|[b,T]\|_{S^{p,q}(L^2(W))}
\lesssim
\|b\|_{\mathrm{Osc}^{p,q}}.
\]
\end{proposition}
\begin{proof}
Since $p>1$ and $q\geq1$, $S^{p,q}$ is a Banach operator ideal. We apply the
dyadic representation theorem to $T$ and estimate the commutator of $b$ with each
term in the representation.

The finite-rank term is controlled by Lemma~\ref{T1lemma}. For the paraproduct
terms, Proposition~\ref{4.1} gives
\[
\|[b,\Pi_{T1}]\|_{S^{p,q}(L^2(W))}
\lesssim
\|b\|_{\mathrm{Osc}^{p,q}(\mathscr D)}\,\|T1\|_{\mathrm{BMO}(\mathscr D)}
\lesssim
\|b\|_{\mathrm{Osc}^{p,q}(\mathscr D)},
\]
and similarly
\[
\|[b,\Pi^*_{T^*1}]\|_{S^{p,q}(L^2(W))}
\lesssim
\|b\|_{\mathrm{Osc}^{p,q}(\mathscr D)},
\]
because Calder\'on--Zygmund operators map $L^\infty$ to $\mathrm{BMO}$, and the
same is true of $T^*$.

It remains to sum the dyadic-shift contributions. Theorem~\ref{5.2} gives, for
fixed $m$,
\begin{align*}
&\|[b,\mathbb S_{m,m}]\|_{S^{p,q}(L^2(W))}
 +\sum_{i=0}^{m-1}\bigl(\|[b,\mathbb S_{m,i}]\|_{S^{p,q}(L^2(W))}
 +\|[b,\mathbb S_{i,m}]\|_{S^{p,q}(L^2(W))}\bigr) \\
&\quad +\sum_{j=1}^{m}\bigl(\|[b,\mathbb S_{0,j}]\|_{S^{p,q}(L^2(W))}
 +\|[b,\mathbb S_{j,0}]\|_{S^{p,q}(L^2(W))}\bigr) \\
&\lesssim \|b\|_{\mathrm{Osc}^{p,q}}
\begin{cases}
(1+m)^{1/p}\delta^{-m\Delta(1/p-1/2)}, & 1<p<2,\\(1+m)^{1/2}\delta^{-m\Delta\varepsilon},&p=2\\[2mm]
(1+m)^{1+1/p}, & p>2.
\end{cases}
\end{align*}
Here we used $m\wedge i=i$, $0\wedge j=0$, and the elementary estimates for the
resulting geometric or polynomial sums. If the space has positive lower
dimension, the factors $(1+i\wedge j)^{1/p}$ from Theorem~\ref{5.2} are omitted,
and the following argument is unchanged.

Since the kernel has smoothness exponent $\eta$, the representation theorem has
coefficients $\omega(\delta^m)=\delta^{m\eta}$. Hence the shift part is bounded by
\[
\|b\|_{\mathrm{Osc}^{p,q}}
\sum_{m=0}^\infty \delta^{m\eta}
\begin{cases}
(1+m)^{1/p}\delta^{-m\Delta(1/p-1/2)}, & 1<p<2,\\(1+m)^{1/2}\delta^{-m\Delta\varepsilon},&p=2\\[2mm](1+m)^{1+1/p}, &p>2.
\end{cases}
\]
For $p>2$, this series converges for every $\eta>0$. 
For $p=2$, this is still the case, provided that we take $\varepsilon\in(0,\eta)$, as we may, since  $\varepsilon>0$ can be taken as small as we like.
For $1<p<2$, we have convergence under the condition $\eta>\Delta(1/p-1/2)$, equivalently
$p>(\eta/\Delta+1/2)^{-1}$. In each case, this is exactly $p>p(\eta)$. Combining the finite-rank,
paraproduct, and shift estimates proves the proposition.
\end{proof}

\begin{proof}[Proof of Theorem~\ref{main}]
Let $p\in(1,\infty)$ and $q\in[1,\infty]$. The lower estimate
\[
\|b\|_{\mathrm{Osc}^{p,q}}
\lesssim
\|[b,T]\|_{S^{p,q}(L^2(W))}
\]
is Proposition~\ref{lowerbound}. If $p>p(\eta)$, the upper estimate
\[
\|[b,T]\|_{S^{p,q}(L^2(W))}
\lesssim
\|b\|_{\mathrm{Osc}^{p,q}}
\]
is Proposition~\ref{upperbound}. The two assertions together give the claimed
equivalence in the range $p>p(\eta)$.
\end{proof}

\section{Declarations}

\subsection*{AI declaration:} The authors acknowledge the use of AI models in helping polish the paper. No mathematical content, results or proofs came from the AI.

\subsection*{Data availability:} There is no data associated with this paper.

\subsection*{Conflicts of interest:} The authors have no conflicts of interest to declare.


\begin{thebibliography}{10}
	


\bibitem{BL1976}
J.~Bergh, J.~L\"ofstr\"om, \emph{Interpolation spaces. An introduction}, Grundlehren der Mathematischen Wissenschaften, No.~223, Springer-Verlag, Berlin--New York, 1976.

\bibitem{bloom} S.~Bloom, \emph{A commutator theorem and weighted {BMO}}, Trans. Amer. Math. Soc. \textbf{292} (1985), 103--122.





\bibitem{Christ2001}
M.~Christ, M.~Goldberg, \emph{Vector {$A_2$} weights and a
	{H}ardy-{L}ittlewood maximal function}, Trans. Amer. Math. Soc. \textbf{353} (2001), no.~5, 1995--2002.



\bibitem{Coifman1}R.~Coifman, R.~Rochberg, G.~Weiss, \emph{Factorization theorems for {H}ardy spaces in several variables}, Ann. of Math. \textbf{103} (1976), no.~3, 611--635.

\bibitem{fan1}
Z.~Fan, M.~Lacey, J.~Li, and X.~Xiong, \emph{Schatten--Lorentz characterization of Riesz transform commutator associated with {B}essel operators}, J. Funct. Anal. \textbf{290} (2026), no. 2, Paper No. 111233, 52 pp.

\bibitem{Fan2} Z.~Fan, J.~Li, E.~McDonald, F.~Sukochev, and D.~Zanin, \emph{Endpoint weak {S}chatten class estimates and trace formula for commutators of {R}iesz transforms with multipliers on {H}eisenberg groups}, J. Funct. Anal. \textbf{286} (2024), no.~1, Paper No.~110188, 72 pp.



\bibitem{GLW}
Z.~Gong, J.~Li, B.~D.~Wick, \emph{Besov spaces, Schatten classes and weighted versions of the quantised derivative}, Anal. Math. \textbf{49} (2023), no.~4, 971--1006.

\bibitem{Hajlasz}
P.~Haj\l{}asz, \emph{Sobolev spaces on an arbitrary metric space.}
Potential Anal. \textbf{5} (1996), no.~4, 403--415.

\bibitem{HytonenA2}
T.~Hyt\"{o}nen, \emph{The sharp weighted bound for general
	{C}alder\'{o}n-{Z}ygmund operators}, Ann. of Math. (2) \textbf{175} (2012), no.~3, 1473--1506.



	\bibitem{Hyto} T.~Hyt\"{o}nen, \emph{Schatten properties of commutators on metric spaces}, J. Funct. Anal. \textbf{291} (2026), no.~7, Paper No.~111556, 98 pp.

\bibitem{HvNVW2}
T.~Hyt\"onen, J.~van Neerven, M.~Veraar, L.~Weis, \emph{Analysis in {B}anach spaces. {V}ol.~{II}. {P}robabilistic methods and operator theory}, Ergebnisse der Mathematik und ihrer Grenzgebiete, vol.~67, Springer, Cham, 2017.



	\bibitem{Hyto1}
T.~Hyt\"{o}nen, R.~Korte, \emph{Characterisations of Sobolev spaces and constant functions over metric spaces }, arXiv:2508.07801, 2025.


\bibitem{HW}
T.~Hyt\"onen, L.~Wu, \emph{Schatten properties of commutators of fractional integrals on spaces of homogeneous type}, arXiv:2511.02686, 2025.

\bibitem{Isralowitz}
J.~Isralowitz, H.~Kwon, S.~Pott, \emph{Matrix weighted norm inequalities for commutators and paraproducts with matrix symbols}, J. Lond. Math. Soc. \textbf{96} (2017), no.~1, 243--270.

\bibitem{Isralowitz2022}
J.~Isralowitz, S.~Pott, S.~Treil, \emph{Commutators in the two scalar and matrix weighted setting}, J. Lond. Math. Soc. (2) \textbf{106} (2022), no.~1, 1--26.



\bibitem{janson}
S.~Janson and J.~Peetre, \emph{Paracommutators--boundedness and {S}chatten--von {N}eumann properties}, Trans. Amer. Math. Soc. \textbf{305} (1988), no.~2, 467--504.
\bibitem{janson2}
S.~Janson and T.~H.~Wolff, \emph{{S}chatten classes and commutators of singular integral operators}, Ark. Mat. \textbf{20} (1982), no.~2, 301--310.

\bibitem{kree1967}
P.~Kr\'ee, \emph{Interpolation d'espaces vectoriels qui ne sont ni norm\'es, ni complets. {A}pplications}, Ann. Inst. Fourier (Grenoble) \textbf{17} (1967), 137--174.

\bibitem{lacey}
M.~Lacey, J.~Li, and B.~D.~Wick, \emph{{S}chatten classes and commutators in the two weight setting, {I}: {H}ilbert transform}, Potential Anal. \textbf{60} (2024), no.~2, 875--894.


\bibitem{lacey3}
M.~Lacey, J.~Li, B.~D.~Wick, and L.~Wu, \emph{{S}chatten classes and commutators in the two weight setting}, J. Funct. Anal. \textbf{289} (2025), no. 6, Paper No. 111028, 51 pp.



\bibitem{LLSW}
J.~Li, C.-W.~Liang, C.-Y.~Shen, B.~D.~Wick, \emph{Commutators with two matrix weights}, arXiv:2609.17054, 2026.

\bibitem{LXY}
J. Li, X. Xiong and F. Yang, \emph{Schatten properties of Calder\'on--Zygmund singular integral commutator
on stratified Lie groups}, J. Math. Pures Appl. (9) \textbf{188} (2024), 73--113.

\bibitem{merucci}
C.~Merucci, \emph{Interpolation dans $C^\omega(H)$}, C.~R. Acad. Sci. Paris S\'er. A-B \textbf{274} (1972), A1163--A1166.



\bibitem{peller}
V.~Peller, \emph{Hankel operators and their applications}, Springer Monographs in Mathematics, Springer-Verlag, New York, 2003.




\bibitem{Rochberg1}
R.~Rochberg, S.~Semmes, \emph{A decomposition theorem for {BMO} and applications}, J. Funct. Anal. \textbf{67} (1986), no.~2, 228--263.

\bibitem{Rochberg}
R.~Rochberg, S.~Semmes, \emph{Nearly weakly orthonormal sequences, singular value estimates, and {C}alder\'on--{Z}ygmund operators}, J. Funct. Anal. \textbf{86} (1989), no.~2, 237--306.

\bibitem{square2}
S.~Treil, \emph{Mixed $A_2$-$A_\infty$ estimates of the nonhomogeneous vector square function with matrix weights},
Proc. Amer. Math. Soc., \textbf{151} (2023), 3381--3389.


\bibitem{Treil1997}
S.~Treil, A.~Volberg, \emph{Wavelets and the angle between past and future}, J. Funct. Anal. \textbf{143} (1997), no.~2, 269--308.


\bibitem{wei}
Z.~Wei, H.~Zhang, \emph{On the boundedness and Schatten class property of noncommutative martingale paraproducts and operator-valued commutators}, arXiv:2407.14988, 2024.

\end{thebibliography}
\end{document}